\documentclass[11pt]{amsart}
\usepackage{amsmath}
\usepackage[dvips]{epsfig}
\theoremstyle{plain}
\newtheorem{theorem}{Theorem}[section]
\newtheorem{lemma}[theorem]{Lemma}
\newtheorem{prop}[theorem]{Proposition}
\newtheorem{cor}[theorem]{Corollary}

\theoremstyle{definition}
\newtheorem{remark}[theorem]{Remark}

\newtheorem{example}[theorem]{Example}
\numberwithin{equation}{section}
\begin{document}

\title[The Generalized Jang Equation]
{Existence and Blow-Up Behavior for Solutions\\
of the Generalized Jang Equation}

\author[Han]{Qing Han}
\address{Department of Mathematics\\
University of Notre Dame\\
Notre Dame, IN 46556} \email{qhan@nd.edu}

\address{Beijing International Center for Mathematical Research\\
Peking University\\
Beijing, 100871, China} \email{qhan@math.pku.edu.cn}

\author[Khuri]{Marcus Khuri}
\address{Department of Mathematics\\
Stony Brook University\\ Stony Brook, NY 11794}
\email{khuri@math.sunysb.edu}
\thanks{The first author acknowledges the support of NSF
Grant DMS-1105321. The second author acknowledges the support of
NSF Grant DMS-1007156 and a Sloan Research Fellowship. MSC Codes: 35Q75, 35J70, 83C99, 58J32}
%\date{\today}
\begin{abstract}
The generalized Jang equation was introduced in an attempt to prove
the Penrose inequality in the setting of general initial data
for the Einstein equations. In this paper we give an extensive
study of this equation, proving existence, regularity, and blow-up
results. In particular, precise asymptotics for the blow-up behavior
are given, and it is shown that blow-up solutions are not unique.
\end{abstract}
\maketitle

\section{Introduction}\label{sec1}

In 1978 the physicist P.-S. Jang \cite{Jang} introduced a quasilinear
elliptic equation in connection with the positive energy
conjecture of General Relativity.
Since then it has had numerous applications from quasilocal mass to
the existence of black holes. An excellent survey of these applications
may be found in \cite{AnderssonEichmairMetzger}.
However, Jang's equation has been shown to be unsuitable for an
application to the Penrose inequality \cite{MalecOMurchadha}, and
for this reason a generalization of this equation has
been proposed in \cite{BrayKhuri} and \cite{BrayKhuri1}. Further
applications may be found in \cite{DisconziKhuri} and \cite{KhuriWeinstein}. The purpose of
this paper is to give a complete analysis of the generalized
Jang equation, in a setting suitable for the Penrose inequality.

Consider an initial data set $(M,g,k)$ for the Einstein equations.
This consists of a Riemannian 3-manifold $M$ with metric $g$,
and a symmetric 2-tensor $k$, which satisfy the constraint
equations:
\begin{align}\label{1}
\begin{split}
2\mu &= R+(Trk)^{2}-|k|^{2},\\
J &= div(k+(Trk)g).
\end{split}
\end{align}
Here $\mu$ and $J$ are the energy and momentum densities of
the matter fields, respectively, and $R$ is the scalar curvature
of $g$. If all measured energy densities are nonnegative then
$\mu\geq|J|$, which will be referred to as the dominant energy
condition. This condition may be viewed as
a quasi nonnegativity of the scalar curvature. In fact in the
time symmetric ($k=0$) or maximal case ($Trk=0$), one does have
$R\geq 0$, and this condition plays a central role in the proof
of several results such as the positive energy theorem
\cite{SchoenYau0}, \cite{Witten} and the Penrose inequality
\cite{Bray}, \cite{HuiskenIlmanen}. Moreover, the primary difficulty
in establishing these
results for general initial data is the lack of nonnegative scalar
curvature. Thus, in the general case, one is motivated to deform
the initial data in an appropriate way (depending on the problem
at hand) such that the dominant energy condition yields nonnegativity
of the scalar curvature for the deformed metric. The deformation
chosen by Jang \cite{Jang} is given by $\overline{g}=g+df^{2}$,
for some function $f$ defined on $M$. Notice that $\overline{g}$
is the induced metric on the graph $\Sigma=\{t=f(x)\}$ in the
product 4-manifold $(M\times\mathbb{R},g+dt^{2})$. When
viewed in this way, the optimal Jang deformation is given by a
hypersurface $\Sigma\subset(M\times\mathbb{R},g+dt^{2})$ which
has scalar curvature that is as nonnegative as possible. It turns
out that this optimal deformation arises as the hypersurface
which satisfies Jang's equation
\begin{equation}\label{2}
H_{\Sigma}-Tr_{\Sigma}k=0,
\end{equation}
where $H_{\Sigma}$ is mean curvature and $Tr_{\Sigma}k$ denotes the
trace of $k$ (extended trivially to the 4-manifold) over $\Sigma$.
While this deformation does not necessarily yield
nonnegative scalar curvature, it is sufficient to make a further
conformal deformation to zero scalar curvature, as was carried out
by Schoen and Yau \cite{SchoenYau} in their proof of the positive
energy theorem.
Jang's approach also beautifully handles the case of equality for
the positive energy theorem. In the case of equality, the Jang
surface $(\Sigma,\overline{g})$ is isometric to Euclidean space
$(\mathbb{R}^{3},\delta)$, and hence $g=\delta-df^{2}$. It follows
that the map $x\mapsto (x,f(x))$ yields an isometric embedding of
$(M,g)$ into the Minkowski spacetime $\mathbb{M}^{4}$; it can also
be shown that the second
fundamental form of this embedding agrees with $k$, as desired. In
some sense, Jang's procedure seems to be tailor made for the
positive energy theorem. In fact it is so well calibrated to the
positive energy theorem,
that this method is rendered inapplicable to the Penrose inequality.
The reason for this is as follows. As explained in \cite{BrayKhuri}
and \cite{BrayKhuri1}, in the case of equality for the Penrose
inequality, the Jang
metric $\overline{g}$ should coincide with $g_{sc}$, the induced
metric on the $t=0$ slice of the Schwarzschild spacetime
$\mathbb{SC}^{4}$. It follows that $g=g_{sc}-df^{2}$, and the map
$x\mapsto (x,f(x))$ cannot yield an embedding of $(M,g)$ into
$\mathbb{SC}^{4}$, since the Schwarzschild metric has the warped
product structure $-\phi_{sc}^{2}dt^{2}+g_{sc}$. However, this
observation leads to a natural modification of the Jang approach
from which
the generalized Jang equation arises.

We now search for a hypersurface $\Sigma$, given by a graph $t=f(x)$,
inside the warped product space $(M\times\mathbb{R},g+\phi^{2}dt^{2})$,
where $\phi$ is a nonnegative function defined on $M$. For certain
applications
the choice of $\phi$ may depend on $f$, however in this paper we will
assume that $\phi$ is fixed independent of $f$. The goal or motive
which leads to the generalized Jang equation is the same as in the
classical case.
That is, we search for a hypersurface which has the most positive
scalar curvature that is possible. In order to have any chance of
obtaining
a positivity property for the scalar curvature, we would like the
Jang surface $\Sigma$ to
satisfy an
equation with the same structure as in \eqref{2}, namely
\begin{equation}\label{3}
H_{\Sigma}-Tr_{\Sigma}K=0,
\end{equation}
where again $H_{\Sigma}$ is mean curvature and $Tr_{\Sigma}K$
denotes the trace of $K$ over $\Sigma$. Here, however, $K$
represents a nontrivial extension of the initial data $k$ to
all of $M\times\mathbb{R}$. In (\ref{2}) $k$ is extended
trivially in that $k(\partial_{t},\cdot)=0$, but it turns out
that the trivial extension is not appropriate for a Jang surface
inside the warped product
metric. This is due to the fact that for applications it is
desirable for solutions of the generalized Jang equation to blow-up
at apparent horizons, and as is shown in \cite{BrayKhuri},
this is possible when $K$ is extended as follows:
\begin{align}\label{4}
\begin{split}
K(\partial_{x^{i}},\partial_{x^{j}})=K(\partial_{x^{j}},\partial_{x^{i}})
&=k(\partial_{x^{i}},\partial_{x^{j}})\text{ }\text{ }\text{
}\text{ for }\text{ }\text{ }\text{ }1\leq i,j\leq
3,\\
K(\partial_{x^{i}},\partial_{t})=K(\partial_{t},\partial_{x^{i}})
&= 0\text{ }\text{ }\text{ }\text{ }\text{ }\text{ }\text{
}\text{ }\text{ }\text{ }\text{ }\text{ }\text{ }\text{ }\text{
}\text{ for }\text{ }\text{ }\text{ }1\leq i\leq
3,\\
K(\partial_{t},\partial_{t})&=
\frac{\phi^{2}g(\nabla f,\nabla\phi)}{\sqrt{1+\phi^{2}|\nabla f|^{2}}},
\end{split}
\end{align}
where $x^{i}$, $i=1,2,3$, are local coordinates on $M$. Moreover
this particular extension also yields an optimal positivity property
for the scalar curvature of solutions to \eqref{3}.
Namely, a long calculation (\cite{BrayKhuri}, \cite{BrayKhuri1})
gives the following formula for the scalar curvature
of $\Sigma$ satisfying equation \eqref{3} with extension \eqref{4}:
\begin{equation}\label{5}
\overline{R}=2(\mu-J(w))+
|A-K|_{\Sigma}|^{2}+2|q|^{2}
-2\phi^{-1}\overline{div}(\phi q).
\end{equation}
Here $\overline{g}=g+\phi^{2}df^{2}$ and $A$ are the induced metric
and second fundamental form of $\Sigma$, respectively, $K|_{\Sigma}$ is
the restriction to $\Sigma$ of the extended tensor $K$, $\overline{div}$
is the divergence operator with respect to $\overline{g}$, and $q$ and
$w$ are
1-forms given by
\begin{equation*}\label{6}
w_{i}=\frac{\phi f_{i}}{\sqrt{1+\phi^{2}|\nabla f|^{2}}},\text{
}\text{ }\text{ }\text{ }
q_{i}=\frac{\phi f^{j}}{\sqrt{1+\phi^{2}|\nabla f|^{2}}}(A_{ij}-(K|_{\Sigma})_{ij}),
\end{equation*}
with $f^{j}=g^{ij}f_{i}$. If the dominant energy condition is satisfied,
then all terms appearing on the right-hand side of \eqref{5} are
nonnegative, except possibly the last term.
However the last term has a special structure, and in many applications
it is clear that a specific choice
of $\phi$ will allow one to `integrate away' this divergence term, so
that in effect the scalar curvature
is weakly nonnegative (that is, nonnegative when integrated against
certain functions).

When the tensor $k$ is extended according to \eqref{4}, we will
refer to equation \eqref{3} as the generalized Jang equation, and
the solution $\Sigma=\{t=f(x)\}$ will be called the Jang surface.
In local coordinates the generalized Jang equation takes the
following form:
\begin{equation}\label{7}
\left(g^{ij}-\frac{\phi^{2}f^{i}f^{j}}{1+\phi^{2}|\nabla f|^{2}}\right)
\left(\frac{\phi\nabla_{ij}f+\phi_{i}f_{j}+\phi_{j}f_{i}}
{\sqrt{1+\phi^{2}|\nabla f|^{2}}}-k_{ij}\right)=0.
\end{equation}
This is a quasilinear elliptic equation, which degenerates
when $f$ blows-up or $\phi=0$. When $\phi=1$ this reduces to the classical
Jang equation studied by Schoen and Yau \cite{SchoenYau}.

Our study of \eqref{7} is naturally divided into two steps. In the first step, the setting consists of
a complete initial data set on which the function $\phi$ is strictly positive.
We will show that an analogue of the existence and
regularity result, as obtained by Schoen and Yau \cite{SchoenYau} for the classical Jang equation, holds in this case.
The proof follows the same general ideas present in \cite{SchoenYau}
with appropriate modification. The primary difference is that (especially for
applications) it is important
to track precisely how the estimates depend on $\phi$ and its
derivatives. These estimates will also be helpful
for the case in which $\phi$
is allowed to vanish. In the second and primary step, the setting consists of an initial data set with outermost
apparent horizon boundary along which $\phi$ vanishes. Our study in this case constitutes the main focus and purpose of this paper.
We will establish the existence of solutions which blow up
along the boundary, and will give precise descriptions of the blow-up rates by constructing upper and lower barriers. The fact that $\phi$
vanishes on the boundary adds an extra degeneracy to equation \eqref{7}, which makes the analysis in this case much more difficult.
The existence of blow-up solutions for the classical Jang equation has been proven
by Metzger in \cite{Metzger} (see also \cite{EichmairMetzger}).

In this paper we will always assume that the initial data are
asymptotically flat (with one end),
so that at spatial infinity the metric and extrinsic curvature
satisfy the following fall-off conditions
\begin{equation*}\label{8}
|\partial^{m}(g_{ij}-\delta_{ij})|=O(|x|^{-m-1}),\text{ }\text{
}\text{ }\text{ }|\partial^{m}k_{ij}|=O(|x|^{-m-2}),\text{ }\text{
}\text{ }m=0,1,2,\text{ }\text{ }\text{ as }\text{ }\text{
}|x|\rightarrow\infty.
\end{equation*}
Moreover if the initial data has boundary, it will be assumed to
consist of an outermost apparent horizon.
To explain this more precisely, recall that the strength of the
gravitational field in the vicinity of a
2-surface $S\subset M$ may be measured by the null expansions
\begin{equation}\label{9}
\theta_{\pm}:=H_{S}\pm Tr_{S}k,
\end{equation}
where $H_{S}$ is the mean curvature with respect to the unit
outward normal (pointing towards spatial infinity). The null
expansions measure the rate of change of area for a shell of light
emitted by the surface in the outward future direction
($\theta_{+}$), and outward past direction ($\theta_{-}$).  Thus
the gravitational field is interpreted as being strong near
$S$ if $\theta_{+}< 0$ or $\theta_{-}< 0$, in which case
$S$ is referred to as a future (past) trapped surface.
Future (past) apparent horizons arise as boundaries of future
(past) trapped regions and satisfy the equation $\theta_{+}=0$
($\theta_{-}=0$). In the setting of the initial data set
formulation of the Penrose inequality, apparent horizons take the
place of event horizons, in that the area of the event horizon is
replaced by the least area required to enclose the outermost apparent
horizon.  Here, an outermost future (past) apparent horizon refers to a future (past) apparent
horizon outside of which there is no other apparent horizon; such
a horizon may have several components, each having spherical
topology (\cite{Galloway}, \cite{GallowaySchoen}). In this paper we will refer to the union of the
outermost future apparent horizon and outermost past apparent horizon as the \textit{outermost apparent horizon},
and we will assume that past and future horizon components do \textit{not} intersect.
For the sake of applications it is desirable for solutions of the
generalized Jang equation to blow-up to $+\infty$ ($-\infty$) in
the form of a cylinder over
the future (past) components of an outermost apparent horizon, and
to vanish at spatial infinity. The fall-off rate for the solution
of the generalized
Jang equation depends on the asymptotics of the warping factor,
however here we will always assume that
\begin{equation}\label{10}
\phi(x)=1+\frac{C}{|x|}+O\left(\frac{1}{|x|^{2}}\right)
\text{ }\text{
}\text{ as }\text{ }\text{ }|x|\rightarrow\infty
\end{equation}
for some constant $C$, which yields
\begin{equation}\label{11}
|\nabla^{m}f|(x)=O(|x|^{-\frac{1}{2}-m})\text{ }\text{ }\text{
as }\text{ }\text{ }|x|\rightarrow\infty,\text{ }\text{ }\text{
}\text{ }m=0,1,2.
\end{equation}
These asymptotics ensure that the ADM energy of the Jang surface
agrees with that of the initial data.

As mentioned above, our primary goal is to study \eqref{7} in the setting of initial data with outermost
horizon boundary along which $\phi$ vanishes, and to describe the blow-up behavior of its solutions
near the boundary. To this end, we will need to know the asymptotics of
the warping factor $\phi$ and of the null expansions. That is, we assume that the warping factor
and the null expansions vanish in a
controlled way at the horizon. Set $\tau(x)=dist(x,\partial M)$.
In a neighborhood of $\partial M$, we impose the following structure
condition
\begin{equation}\label{14}
\phi(x)=\tau^{b}(x)\widetilde{\phi}(x),
\end{equation}
for some smooth (up to the boundary) strictly positive function $\widetilde{\phi}$,
with $b\geq 0$.
Next, we let $S_{\tau}$ denote level sets of
the geodesic flow emanating from $\partial M$, that is, each point
of $S_{\tau}$ is of distance $\tau$ from the boundary. Furthermore
decompose $\partial M=\partial_{+}M\cup\partial_{-}M$, where $\partial_{+}M$
($\partial_{-}M$) denotes the future (past) apparent horizon components.
We then stipulate that near $\partial_{\pm} M$
\begin{equation}\label{15}
|\theta_{\pm}(S_{\tau})|\leq c\tau^{l}
\end{equation}
for some constants $l,c>0$. Notice that when the initial data are smooth up to the boundary, as will always be
the case in this paper, $l\geq 1$. Whether or not $f$ actually blows up and approximates a cylinder over the horizon is highly
dependent on the relationship between the vanishing rates of
$\theta_{\pm}$ and $\phi$, that is, the relation between $l$ and $b$.
In some cases we require a more restrictive condition
\begin{equation}\label{12}
c^{-1}\tau^{l}\leq\theta_{\pm}(S_{\tau})\leq c\tau^{l},
\end{equation}
with constants $l,c>0$.

The main result in this paper is the following theorem.

\begin{theorem}\label{thm4}
Suppose that $(M,g,k)$ is a smooth, asymptotically flat initial data
set, with outermost apparent horizon boundary $\partial M=\partial_{+}M\cup\partial_{-}M$.
Suppose further that $\phi$ is smooth, strictly positive away from $\partial M$, and satisfies
\eqref{10} and \eqref{14} for some $b\geq 0$.

$\operatorname{(1)}$ If $-\frac{l-1}2\le b<\frac{l+1}2$ and \eqref{12} is valid for some $l\geq 1$,
then there exists a smooth solution $f$ of the generalized
Jang equation \eqref{7}, satisfying \eqref{11}, and such that
$f(x)\rightarrow\pm\infty$ as $x\rightarrow\partial_{\pm}M$. More
precisely, in a neighborhood of $\partial_{\pm}M$
\begin{align}\label{16a}
\begin{split}
\alpha^{-1}\tau^{-b-\frac{l-1}{2}}+\beta^{-1} &\leq \pm f \leq \alpha\tau^{-b-\frac{l-1}{2}}+\beta\text{ }\text{ }\textit{ if }\text{ }\text{ }
-\frac{l-1}{2}< b<\frac{l+1}{2},\\
-\alpha^{-1}\log\tau+\beta^{-1} &\leq \pm f
\leq -\alpha\log\tau+\beta\text{ }\text{ }\textit{ if }\text{ }\text{ }b=-\frac{l-1}{2},
\end{split}
\end{align}
for some positive constants $\alpha$ and $\beta$.

$\operatorname{(2)}$ If $\frac{1}2\le b<\frac{l+1}2$ and \eqref{15} is
valid for some $l\geq 1$, then there exists a smooth solution $f$ of the generalized
Jang equation \eqref{7}, satisfying \eqref{11}, and such that
$f(x)\rightarrow\pm\infty$ as $x\rightarrow\partial_{\pm}M$. More
precisely, in a neighborhood of $\partial_{\pm}M$
\begin{align}\label{16b}
\begin{split}
\alpha^{-1}\tau^{1-2b}+\beta^{-1} &\leq \pm f
\leq \alpha\tau^{1-2b}+\beta\text{ }\text{ }\textit{ if }\text{ }\text{ }
\frac{1}{2}< b<\frac{l+1}{2},\\
-\alpha^{-1}\log\tau+\beta^{-1} &\leq \pm f
\leq -\alpha\log\tau+\beta\text{ }\text{ }\textit{ if }\text{ }\text{ }b=\frac{1}{2},
\end{split}
\end{align}
for some positive constants $\alpha$ and $\beta$.
\end{theorem}

\begin{remark}
The hypothesis $b\geq-\frac{l-1}{2}$ is trivially satisfied under the primary assumptions of this paper, that is $b\geq 0$ and $l\geq 1$. However it is believed
that this theorem, and others having this hypothesis, continue to hold even when these primary assumptions are relaxed. For this reason, the inequality $b\geq-\frac{l-1}{2}$
is included in the statement of such results.
\end{remark}

We note that
there are at least two solutions with different asymptotics when
$\frac{1}{2}\leq b<\frac{l+1}{2}$. The estimates \eqref{16a}
and \eqref{16b} are established by constructing appropriate
sub and super solutions, however these estimates may not hold
when the particular inequalities between $l$ and $b$ are not satisfied.
Furthermore, we exhibit by example, a solution which does not have
cylindrical asymptotics near the outermost apparent horizon when $b=(l+1)/2$.
%and the purity condition does not hold.

It should be pointed out that the lower bounds in \eqref{16a} are new even for the classical Jang equation, that is,
when $\phi=1$.

Theorem \ref{thm4} establishes the existence of solutions to the generalized Jang
equation which possess appropriate behavior for application to the
Penrose inequality. In particular, for an appropriate choice of $\phi$,
the blow-up rates \eqref{16b} show that the Jang
surface is a manifold
with boundary, and that the boundary is a minimal surface.
This allows techniques from the time symmetric proof of the Penrose
inequality to be applied, and is different from the behavior of the
blow-up solutions of the classical Jang equation which possess an
infinitely long neck at the horizon.

The assumption on $\phi$, that it is positive away from the boundary and satisfies \eqref{14}, may seem restrictive if one hopes to apply our
result to the Penrose inequality as described in \cite{BrayKhuri1}. However this is not the case. Three proposals for a coupling of the generalized Jang
equation with other equations were outlined in \cite{BrayKhuri1}, and it is the coupling with Bray's conformal flow which produces a warping factor $\phi$ that
satisfies our hypotheses. The one missing ingredient here, in this paper, is that $\phi$ is fixed, and does not depend on $f$. Nevertheless, Theorem \ref{thm4} provides the
important first step in the difficult problem of analyzing the coupled Jang/conformal flow system. More precisely, given $\phi_{0}$ satisfying \eqref{14}, Theorem \ref{thm4} produces $f_{0}$,
from which we may construct the Jang metric $\overline{g}_{0}=g+\phi_{0}^{2}df_{0}^{2}$. According to \cite{BrayKhuri1}, Bray's conformal flow with respect to $\overline{g}_{0}$
produces a new warping factor $\phi_{1}$, which will also satisfy the hypotheses of Theorem \ref{thm4}. We may then continue this procedure indefinitely to produce
sequences of functions $\{\phi_{i}\}$, $\{f_{i}\}$. The second and final step should entail establishing appropriate a priori estimates, to show that a subsequence converges to yield
a solution of the coupled system of equations. This last step will be quite involved.

This paper is organized as follows. In Section \ref{sec2} we derive
local a priori estimates, and in Section \ref{sec3} an important
Harnack inequality is proved. These two results are then combined with
further global estimates to establish the existence of solutions on complete initial data, in Section
\ref{sec4}. In Section \ref{sec5} we produce the blow-up solutions
of Theorem \ref{thm4}, and construct appropriate super solutions.
Lastly, the more difficult sub solutions are constructed in Section \ref{sec6}
and Theorem \ref{thm4} is proven.

\section{A Priori Estimates for the Generalized Jang Equation}\label{sec2}

In this section we begin the proof of existence for the generalized
Jang equation. Our first goal is to prove the same existence and
regularity result as obtained by Schoen and Yau \cite{SchoenYau},
for arbitrary positive warping factor $\phi$. This is fairly straight forward,
in that the methods of Schoen and Yau still apply with little
modification. The primary difference is that (especially for
applications) it is important
to track precisely how the estimates depend on $\phi$ and its
derivatives. The estimates of this section will also be helpful
for the case in which $\phi$
is allowed to vanish, but for now we assume that $\phi$ is strictly positive.

Let us set the notation. Consider a hypersurface $\Sigma\subset(M\times\mathbb{R},g+\phi^{2}dt^{2})$ given by the
graph of a function $t=f(x)$.
%For the time
%being we do not assume that this surface satisfies the generalized Jang
%equation, although we will still refer to it as the Jang surface.
This surface has
induced metric $\overline{g}=g+\phi^{2}df^{2}$, and inverse matrix
\begin{equation*}
\overline{g}^{ij}=g^{ij}-\frac{\phi^{2}f^{i}f^{j}}{1+\phi^{2}|\nabla f|^{2}},
\end{equation*}
where $f^{i}=g^{ij}f_{j}$ and $f_{i}=\partial_{x^{i}}f$,
with $x^{i}$, $i=1,2,3$, local coordinates on $M$. Throughout the paper,
a bar will be placed over
geometric quantities associated with $\Sigma$. For instance
$\overline{\nabla}$ will denote the induced connection on $\Sigma$,
whereas $\nabla$ will denote the connection on the ambient manifold $(M\times\mathbb{R},g+\phi^{2}dt^{2})$. The unit normal to $\Sigma$ is given by
\begin{equation*}
N=\frac{\nabla f-\phi^{-2}\partial_{t}}{\sqrt{\phi^{-2}+|\nabla f|^{2}}},
\end{equation*}
and $X_{i}=\partial_{x^{i}}+f_{i}\partial_{t}$, $i=1,2,3$
form a basis for the tangent space. Moreover a calculation (\cite{BrayKhuri}, \cite{BrayKhuri1}) shows that the
second fundamental form of $\Sigma$ is given by
\begin{equation}\label{eq-SecondFF}
A_{ij}=\langle\nabla_{X_{i}}N,X_{j}\rangle=
\frac{\phi\nabla_{ij}f+\phi_{i}f_{j}+\phi_{j}f_{i}
+\phi^{2}f^{m}\phi_{m}f_{i}f_{j}}{\sqrt{1+\phi^{2}|\nabla f|^{2}}}.
\end{equation}
In this section it will be assumed that $\Sigma$ satisfies an equation of the form
\begin{equation}\label{eq-BasicEquation}
H-TrK=F,
\end{equation}
for some function $F$ to be specified, where $H=\overline{g}^{ij}A_{ij}$
is mean curvature and $TrK$ denotes the trace of $K$ over $\Sigma$. In
what follows, as in \eqref{eq-BasicEquation}, we will drop
the subscript $\Sigma$ when denoting the mean curvature and trace
operations with respect to $\Sigma$.
%Moreover, although $\Sigma$ does not necessarily satisfy the generalized
%Jang equation, we will still refer to it as the Jang surface.

We first estimate the mean curvature and its derivatives.

\begin{lemma}\label{lemma-MeanCurvature}
Suppose that the surface $\Sigma$ satisfies \eqref{eq-BasicEquation}.
Then
\begin{equation}\label{eq-EstimateMeanCurv}
|H|\leq |F|+c(1+|\nabla\log\phi|),
\end{equation}
and
\begin{equation}\label{eq-EstimateGradientMeanCurv}
|\overline{\nabla}H|\leq c(1+|\nabla F|+|\nabla\log\phi|^{2}+\phi^{-1}|\nabla^{2}\phi|+|A|+|A||\nabla\log\phi|),
\end{equation}
where $c$ is a universal constant.
\end{lemma}

\begin{proof}
First, by \eqref{eq-BasicEquation}, we have
\begin{equation}\label{eq-MeanCurv}
H=\overline{g}^{ij}\left(k_{ij}
+\frac{\phi^{2}(\phi^{l}f_{l})f_{i}f_{j}}{\sqrt{1+\phi^{2}|\nabla f|^{2}}}\right)
+F
=\overline{g}^{ij}k_{ij}+\frac{\phi^{2}(\phi^{l}f_{l})
|\nabla f|^{2}}{(1+\phi^{2}|\nabla f|^{2})^{3/2}}+F.
\end{equation}
This implies \eqref{eq-EstimateMeanCurv} easily.

Next, a simple differentiation yields
\begin{align*}
\partial_{a}H =& \overline{g}^{ij}k_{ij;a}+\partial_{a}F
+\left(\frac{-2\phi\phi_{a}f^{i}f^{j}
-2\phi^{2}f^{i}\nabla_{a}f^{j}}{1+\phi^{2}|\nabla f|^{2}}\right)k_{ij}\\
& +\frac{\phi^{2}f^{i}f^{j}}{(1+\phi^{2}|\nabla f|^{2})^{2}}(2\phi\phi_{a}
|\nabla f|^{2}+2\phi^{2}f^{n}\nabla_{na}f)k_{ij}
+\frac{2\phi^{2}(\phi^{m}f_{m})f^{n}\nabla_{na}f}{(1+\phi^{2}|\nabla f|^{2})^{3/2}}\\
& +\frac{2\phi\phi_{a}(\phi^{n}f_{n})|\nabla f|^{2}
+\phi^{2}(\nabla_{a}\phi^{n})f_{n}|\nabla f|^{2}
+\phi^{2}\phi^{n}\nabla_{na}f|\nabla f|^{2}}
{(1+\phi^{2}|\nabla f|^{2})^{3/2}}\\
& -\frac{3}{2}\frac{\phi^{2}(\phi^{m}f_{m})|\nabla f|^{2}}{(1+\phi^{2}|\nabla f|^{2})^{5/2}}(2\phi\phi_{a}|\nabla f|^{2}+2\phi^{2}f^{n}\nabla_{na}f).
\end{align*}
However, by \eqref{eq-SecondFF}
\begin{equation*}
\nabla_{ij}f=\phi^{-1}(\sqrt{1+\phi^{2}|\nabla f|^{2}}A_{ij}-\phi_{i}f_{j}-\phi_{j}f_{i}-\phi^{2}(\phi^{m}f_{m})f_{i}f_{j}).
\end{equation*}
Thus, after substitution, some cancelations occur and we obtain
\begin{align*}
\partial_{a}H =& \overline{g}^{ij}k_{ij;a}+\partial_{a}F
+\left(\frac{2\phi\phi_{i}f_{j}f_{a}}{1+\phi^{2}|\nabla f|^{2}}
-\frac{2\phi f_{i}A_{ja}}{\sqrt{1+\phi^{2}|\nabla f|^{2}}}+\frac{2\phi^{3}f_{i}f_{j}f^{n}A_{an}}{(1+\phi^{2}|\nabla f|^{2})^{3/2}}\right)k^{ij}\\
& +\frac{\phi^{2}|\nabla f|^{2}f^{n}\nabla_{an}\phi}{(1+\phi^{2}|\nabla f|^{2})^{3/2}}
-\frac{2\phi(\phi^{m}f_{m})^{2}f_{a}}{(1+\phi^{2}|\nabla f|^{2})^{3/2}}
-\frac{\phi|\nabla f|^{2}|\nabla\phi|^{2}f_{a}}{(1+\phi^{2}|\nabla f|^{2})^{3/2}}-\frac{\phi|\nabla f|^{2}(\phi^{m}f_{m})\phi_{a}}{(1+\phi^{2}|\nabla f|^{2})^{3/2}}\\
& +A_{an}\left(\frac{2\phi(\phi^{m}f_{m})f^{n}}{1+\phi^{2}|\nabla f|^{2}}
+\frac{\phi|\nabla f|^{2}\phi^{n}}{1+\phi^{2}|\nabla f|^{2}}
-\frac{3\phi^{3}|\nabla f|^{2}(\phi^{m}f_{m})f^{n}}
{(1+\phi^{2}|\nabla f|^{2})^{2}}\right).
\end{align*}
This implies \eqref{eq-EstimateGradientMeanCurv}.
\end{proof}

We point out that \eqref{eq-EstimateMeanCurv} yields a uniform
pointwise estimate of the mean curvature. However,
\eqref{eq-EstimateGradientMeanCurv} illustrates that first derivatives of the
mean curvature are estimated in terms of the second fundamental form, which
is not yet controlled. In fact, estimating the second
fundamental form will take up a major part of this section.

We next establish a $C^{0}$ estimate for the second
fundamental form of $\Sigma$. This result will arise from a
Moser iteration applied to the Simons identity. See also the paper \cite{anderssonmetzger}, in which a similar result is proven
in more general ambient geometries.
%First, however, we
%must establish local $L^{4}$ bounds. In order to accomplish this,
%several calculations will follow from which the appropriate inequalities
%will be derived.

\begin{theorem}\label{thm-PointwiseEsti2FF}
Suppose that the surface $\Sigma$ satisfies \eqref{eq-BasicEquation}.
Then
\begin{equation}\label{eq-PointwiseEstimate2FF}
\sup_{\Sigma}|A|\leq c
\end{equation}
where $c$ depends on $|\log\phi|$, $|\nabla\log\phi|$,
$|\nabla^{2}\log\phi|$, $|F|$, and $|N(F)|$.
\end{theorem}

\begin{proof} The proof is lengthy and is divided into several steps.

{\it Step 1. Derivation of Simons identity and several related inequalities.}
Recall the Ricci commutation formula and Codazzi equations:
\begin{equation}\label{eq-RicciCommutation}
\overline{\nabla}_{a}\overline{\nabla}_{n}A_{ij}
-\overline{\nabla}_{n}\overline{\nabla}_{a}A_{ij}
=-A_{mj}\overline{R}^{m}_{\text{ }\text{ }ian}
-A_{im}\overline{R}^{m}_{\text{ }\text{ }jan},
\end{equation}
and
\begin{equation}\label{eq-Codazzi}
\overline{\nabla}_{a}A_{ij}=\overline{\nabla}_{i}A_{aj}+R_{Njia}.
\end{equation}
Use these to obtain
\begin{align*}
\overline{\Delta}A_{ij} &= \overline{g}^{na}\overline{\nabla}_{n}\overline{\nabla}_{a}A_{ij}\\
&=\overline{g}^{na}(\overline{\nabla}_{n}\overline{\nabla}_{i}A_{aj}
+\overline{\nabla}_{n}R_{Njia})\\
&=\overline{g}^{na}(\overline{\nabla}_{i}\overline{\nabla}_{n}A_{aj}
-A_{mj}\overline{R}^{m}_{\text{ }\text{ }ani}
-A_{am}\overline{R}^{m}_{\text{ }\text{ }jni}
+\overline{\nabla}_{n}R_{Njia})\\
&=\overline{g}^{na}\overline{\nabla}_{i}(\overline{\nabla}_{j}A_{na}
+R_{Najn})-\overline{g}^{na}(A_{mj}\overline{R}^{m}_{\text{ }\text{ }ani}+A_{am}\overline{R}^{m}_{\text{ }\text{ }jni}
-\overline{\nabla}_{n}R_{Njia})\\
&=\overline{\nabla}_{i}\overline{\nabla}_{j}H
+\overline{g}^{na}(\overline{\nabla}_{i}R_{Najn}+\overline{\nabla}_{n}R_{Njia}
-A_{mj}\overline{R}^{m}_{\text{ }\text{ }ani}
-A_{am}\overline{R}^{m}_{\text{ }\text{ }jni}).
\end{align*}
Now use the Gauss equations
\begin{equation}\label{eq-GaussEq}
\overline{R}_{ijnd}=R_{ijnd}+A_{in}A_{jd}-A_{id}A_{jn}
\end{equation}
to find
\begin{align*}
\overline{\Delta}A_{ij} =& \overline{\nabla}_{i}\overline{\nabla}_{j}H
+\overline{g}^{na}(\overline{\nabla}_{i}R_{Najn}
+\overline{\nabla}_{n}R_{Njia})
-\overline{g}^{na}\overline{g}^{md}(A_{mj}R_{dani}+A_{am}R_{djni})\\
& -\overline{g}^{na}\overline{g}^{md}A_{mj}(A_{dn}A_{ai}-A_{di}A_{an})
-\overline{g}^{na}\overline{g}^{md}A_{am}(A_{dn}A_{ji}-A_{id}A_{jn}).
\end{align*}
Therefore
\begin{equation*}
\overline{\Delta}A_{ij}=\overline{\nabla}_{i}\overline{\nabla}_{j}H
-|A|^{2}A_{ij}+HA_{im}A_{j}^{m}+G_{ij},
\end{equation*}
where
\begin{equation*}
G_{ij}=\overline{g}^{na}(\overline{\nabla}_{i}R_{Najn}
+\overline{\nabla}_{n}R_{Njia}
-\overline{g}^{md}A_{mj}R_{dani}-\overline{g}^{md}A_{am}R_{djni}).
\end{equation*}
We then obtain
\begin{align}\label{eq-SimonIdentity}
\begin{split}
\frac{1}{2}\overline{\Delta}|A|^{2} =& \overline{g}^{id}\overline{g}^{ja}(A_{da}\overline{\Delta}A_{ij}
+\overline{g}^{mn}\overline{\nabla}_{m}A_{da}\overline{\nabla}_{n}A_{ij})\\
=& \overline{g}^{id}\overline{g}^{ja}(\overline{g}^{mn}
\overline{\nabla}_{m}A_{da}\overline{\nabla}_{n}A_{ij}+A_{da}G_{ij})-|A|^{4}\\
& +\overline{g}^{id}\overline{g}^{ja}(HA_{da}A_{im}A_{j}^{m}
+A_{da}\overline{\nabla}_{i}\overline{\nabla}_{j}H).
\end{split}
\end{align}
This is the Simons identity for $\Sigma$.

Note that we also have
\begin{equation*}
\frac{1}{2}\overline{\Delta}|A|^{2}=|A|\overline{\Delta}|A|
+|\overline{\nabla}|A||^{2},
\end{equation*}
so that
\begin{equation*}
|A|\overline{\Delta}|A| = |\overline{\nabla}A|^{2}
-|\overline{\nabla}|A||^{2}-|A|^{4}
 +A^{ij}(\overline{\nabla}_{i}\overline{\nabla}_{j}H
 +G_{ij}+HA_{im}A_{j}^{m}).
\end{equation*}
Set
\begin{equation}\label{eq-DefinitionT}
T:=|\overline{\nabla}A|^{2}-|\overline{\nabla}|A||^{2}.
\end{equation}
For the remainder of this paragraph we will work in an orthonormal basis. Then
\begin{equation*}
T=\sum_{i,j,n}(\overline{\nabla}_{n}A_{ij})^{2}
-|A|^{-2}\sum_{n}(\sum_{i,j}A_{ij}\overline{\nabla}_{n}A_{ij})^{2},
\end{equation*}
so that
\begin{equation*}
|A|^{2}T=\frac{1}{2}\sum_{i,j,n,d,m}(A_{ij}\overline{\nabla}_{n}A_{dm}
-A_{dm}\overline{\nabla}_{n}A_{ij})^{2}.
\end{equation*}
By setting $n=i$ and $m=j$ and using the Schwarz inequality, we have
\begin{align*}
|A|^{2}T &= \frac{1}{2}\sum_{i,j,d}(A_{ij}\overline{\nabla}_{i}A_{dj}
-A_{dj}\overline{\nabla}_{i}A_{ij})^{2}
+\frac{1}{2}\sum_{i,j,n,d,m\atop (n,m)\neq(i,j)}
(A_{ij}\overline{\nabla}_{n}A_{dm}
-A_{dm}\overline{\nabla}_{n}A_{ij})^{2}\\
&\geq \frac{1}{18}\sum_{d}\left(\sum_{i,j}A_{ij}
\overline{\nabla}_{i}A_{dj}
-\sum_{i,j}A_{dj}\overline{\nabla}_{i}A_{ij}\right)^{2}
+\frac{1}{2}\sum_{i,j,n,d,m\atop (n,m)\neq(i,j)}
(A_{ij}\overline{\nabla}_{n}A_{dm}
-A_{dm}\overline{\nabla}_{n}A_{ij})^{2}.
\end{align*}
Moreover, according to the Codazzi equations \eqref{eq-Codazzi},
\begin{equation*}
\sum_{i,j}A_{ij}\overline{\nabla}_{i}A_{dj}=
\sum_{i,j}A_{ij}\overline{\nabla}_{d}A_{ij}+\sum_{i,j}A_{ij}R_{Njdi}
\end{equation*}
and
\begin{equation*}
\sum_{i,j}A_{dj}\overline{\nabla}_{i}A_{ij}=
\sum_{j}A_{dj}\overline{\nabla}_{j}H+\sum_{i,j}A_{dj}R_{Niji}.
\end{equation*}
It follows that
\begin{align*}
|A|^{2}T \geq& \frac{1}{36}\sum_{l}\left(\sum_{i,j}A_{ij}\overline{\nabla}_{d}A_{ij}\right)^{2}\\
& -\frac{1}{18}\sum_{d}\left(\sum_{i,j}A_{ij}R_{Njdi}-\sum_{j}A_{dj}
\overline{\nabla}_{j}H-\sum_{i,j}A_{dj}R_{Niji}\right)^{2}\\
& +\frac{1}{2}\sum_{i,j,n,d,m\atop (n,m)\neq(i,j)}
(A_{ij}\overline{\nabla}_{n}A_{dm}-A_{dm}
\overline{\nabla}_{n}A_{ij})^{2},
\end{align*}
after using $(a+b)^{2}\geq\frac{1}{2}a^{2}-b^{2}$.
The definition of $T$ in \eqref{eq-DefinitionT} now yields
\begin{align*}
T \geq& \frac{1}{37}\sum_{i,j,n}(\overline{\nabla}_{n}A_{ij})^{2}\\
& -\frac{36}{37}\cdot\frac{1}{18}|A|^{-2}\sum_{d}
\left(\sum_{i,j}A_{ij}R_{Njdi}-\sum_{j}A_{dj}\overline{\nabla}_{j}H
-\sum_{i,j}A_{dj}R_{Niji}\right)^{2}\\
& +\frac{36}{37}\cdot\frac{1}{2}|A|^{-2}\sum_{i,j,n,d,m\atop (n,m)\neq(i,j)}(A_{ij}\overline{\nabla}_{n}A_{dm}-A_{dm}
\overline{\nabla}_{n}A_{ij})^{2}.
\end{align*}
Hence
\begin{align*}
|A|\overline{\Delta}|A| \geq& \frac{1}{37}\sum_{i,j,n}(\overline{\nabla}_{n}A_{ij})^{2}-|A|^{4}
+\sum_{i,j}A_{ij}\left(\overline{\nabla}_{ij}H+G_{ij}
+\sum_{m}HA_{im}A_{jm}\right)\\
& -\frac{2}{37}|A|^{-2}\sum_{d}\left(\sum_{i,j}A_{ij}R_{Njdi}
-\sum_{j}A_{dj}\overline{\nabla}_{j}H-\sum_{i,j}A_{dj}R_{Niji}\right)^{2}\\
& +\frac{2}{37}|A|^{-2}\sum_{i,j,n,d,m\atop (n,m)\neq(i,j)}
(A_{ij}\overline{\nabla}_{n}A_{dm}
-A_{dm}\overline{\nabla}_{n}A_{ij})^{2}.
\end{align*}
We can also write
\begin{equation}\label{eq-EstimateSFF}
|A|\overline{\Delta}|A|\geq\frac{1}{37}\sum_{i,j,n}
(\overline{\nabla}_{n}A_{ij})^{2}-|A|^{4}-|H||A|^{3}
+\sum_{i,j}A_{ij}\overline{\nabla}_{ij}H+G_{1},
\end{equation}
where
\begin{align*}
G_{1}:=&\sum_{i,j}A_{ij}G_{ij}
-\frac{2}{37}|A|^{-2}\sum_{d}\left(\sum_{i,j}A_{ij}R_{Njdi}
-\sum_{j}A_{dj}\overline{\nabla}_{j}H-\sum_{i,j}A_{dj}R_{Niji}\right)^{2}\\
& +\frac{2}{37}|A|^{-2}\sum_{i,j,n,d,m\atop (n,m)\neq(i,j)}
(A_{ij}\overline{\nabla}_{n}A_{dm}
-A_{dm}\overline{\nabla}_{n}A_{ij})^{2}.
\end{align*}

{\it Step 2. An $L^4$-estimate for the second fundamental form.}
Taking two traces of the Gauss equations \eqref{eq-GaussEq} produces
\begin{equation*}
\overline{R}=\overline{g}^{in}\overline{g}^{jm}R_{ijnm}+H^{2}-|A|^{2}.
\end{equation*}
Moreover, equation \eqref{eq-BasicEquation} implies that
\begin{equation*}
\overline{R}+|A|^{2}=\overline{g}^{in}\overline{g}^{jm}R_{ijnm}+(TrK+F)^{2}.
\end{equation*}
Recall the formula for the scalar curvature of surface $\Sigma$ (see \cite{BrayKhuri}, \cite{BrayKhuri1})
\begin{align*}
\overline{R} =& 2(\mu-J(w))+|A-K|^{2}+2|q|^{2}
-\frac{2}{\phi}\overline{div}(\phi q)\\
& +H^{2}-(TrK)^{2}+2K(N,N)(H-TrK)+2N(H-TrK).
\end{align*}
By combining the previous two formulas we then have
\begin{align}\label{eq-Identity2FF}
\begin{split}
& |A|^{2}+|A-K|^{2}+2(\mu-J(w))+2|q|^{2}\\
=&\frac{2}{\phi}\overline{div}(\phi q)-(2TrK+F)F-2K(N,N)F\\
& -2N(F)+\overline{g}^{in}\overline{g}^{jm}R_{ijnm}+(TrK+F)^{2}.
\end{split}
\end{align}
Note that
\begin{equation*}
|A-K|^{2}=|A|^{2}+|K|^{2}-2\langle A,K\rangle,
\end{equation*}
and set
\begin{equation*}
G_{2}:=2\langle A,K\rangle-|K|^{2}-2K(N,N)F-2N(F)
+(TrK)^{2}-2(\mu-J(w))+\overline{g}^{in}\overline{g}^{jm}R_{ijnm}.
\end{equation*}
Then, \eqref{eq-Identity2FF} becomes
\begin{equation}\label{eq-BasicIdentity2FF}
2|A|^{2}=\frac{2}{\phi}\overline{div}(\phi q)-2|q|^{2}+G_{2}.
\end{equation}

Let $\psi\in C^{\infty}_{c}(\Sigma)$.
Multiply \eqref{eq-BasicIdentity2FF} by $\phi\psi^{2}$
and integrate by parts to find
\begin{align*}
\int_{\Sigma}2\phi\psi^{2}|A|^{2} &= \int_{\Sigma}-4\phi\psi\langle\overline{\nabla}\psi,q\rangle
-2\phi\psi^{2}|q|^{2}+\phi\psi^{2}G_{2}\\
&\leq \int_{\Sigma}2\phi|\overline{\nabla}\psi|^{2}+\phi\psi^{2}G_{2}.
\end{align*}
Now replace $\psi$ with $|A|^{p}\psi$ to get
\begin{equation}\label{eq-Integral2FF}
\int_{\Sigma}\phi\psi^{2}|A|^{2+2p}\leq\int_{\Sigma}\phi|
\overline{\nabla}(|A|^{p}\psi)|^{2}+\frac{1}{2}\phi\psi^{2}G_{2}|A|^{2p}.
\end{equation}
Expand and integrate the first term by parts:
\begin{align*}
\int_{\Sigma}\phi|\overline{\nabla}(|A|^{p}\psi)|^{2} &= \int_{\Sigma}\phi|A|^{2p}|\overline{\nabla}\psi|^{2}
+2p\phi\psi|A|^{2p-1}\langle\overline{\nabla}|A|,\overline{\nabla}\psi\rangle
+\phi\psi^{2}|\overline{\nabla}|A|^{p}|^{2}\\
&= \int_{\Sigma}\phi|A|^{2p}|\overline{\nabla}\psi|^{2}
-\frac{1}{2}\phi\psi^{2}\overline{\Delta}|A|^{2p}+\phi\psi^{2}
|\overline{\nabla}|A|^{p}|^{2}
-\psi^{2}\langle\overline{\nabla}\phi,\overline{\nabla}|A|^{2p}\rangle\\
&= \int_{\Sigma}\phi|A|^{2p}|\overline{\nabla}\psi|^{2}-\phi\psi^{2}|A|^{p}
\overline{\Delta}|A|^{p}
-2p\psi^{2}|A|^{2p-1}\langle\overline{\nabla}\phi,\overline{\nabla}|A|\rangle.
\end{align*}
Since
\begin{equation*}
|A|^{p}\overline{\Delta}|A|^{p}=p|A|^{2p-1}\overline{\Delta}|A|
+p(p-1)|A|^{2p-2}|\overline{\nabla}|A||^{2},
\end{equation*}
we then have
\begin{align*}
& \int_{\Sigma}\phi\psi^{2}(|A|^{2p-1}(p\overline{\Delta}|A|+|A|^{3})
+p(p-1)|A|^{2p-2}|\overline{\nabla}|A||^{2})\\
\leq& \int_{\Sigma}\phi|A|^{2p}|\overline{\nabla}\psi|^{2}
-2p\psi^{2}|A|^{2p-1}\langle\overline{\nabla}\phi,\overline{\nabla}|A|\rangle
+\frac{1}{2}\phi\psi^{2}|A|^{2p}G_{2}\\
=& \int_{\Sigma}\phi|A|^{2p}|\overline{\nabla}\psi|^{2}
+2\psi|A|^{2p}\langle\overline{\nabla}\psi,\overline{\nabla}\phi\rangle
+\phi\psi^{2}|A|^{2p}\left(\frac{1}{2}G_{2}
+\phi^{-1}\overline{\Delta}\phi\right)\\
\leq& \int_{\Sigma}2\phi|A|^{2p}|\overline{\nabla}\psi|^{2}
+\phi\psi^{2}|A|^{2p}\left(\frac{1}{2}G_{2}
+\phi^{-2}|\overline{\nabla}\phi|^{2}+\phi^{-1}\overline{\Delta}\phi\right).
\end{align*}
By using \eqref{eq-EstimateSFF} for $|A|(\overline{\Delta}|A|+|A|^{3})$
and setting $p=1$ in the above expression, it follows that
\begin{align*}
\int_{\Sigma}\frac{1}{37}\phi\psi^{2}|\overline{\nabla}A|^{2} \leq& \int_{\Sigma}2\phi|A|^{2}|\overline{\nabla}\psi|^{2}
+\phi\psi^{2}|A|^{2}\left(\frac{1}{2}G_{2}
+\phi^{-2}|\overline{\nabla}\phi|^{2}+\phi^{-1}\overline{\Delta}\phi\right)\\
& +\int_{\Sigma}\phi\psi^{2}(|H||A|^{3}-G_{1}-A^{ij}\overline{\nabla}_{ij}H).
\end{align*}
Integrating by parts and applying the Schwarz inequality then yields
\begin{align*}
\int_{\Sigma}c^{-1}\phi\psi^{2}|\overline{\nabla}A|^{2} \leq& \int_{\Sigma}c\phi|A|^{2}|\overline{\nabla}\psi|^{2}
+\phi\psi^{2}|A|^{2}\left(\frac{1}{2}G_{2}+c|\overline{\nabla}\log\phi|^{2}\right)\\
& +\int_{\Sigma}\phi\psi^{2}(|H||A|^{3}-G_{1}+c|\overline{\nabla}H|^{2}).
\end{align*}
Above and in what follows, $c>0$ will always
denote an appropriately large constant.
Although there are derivatives of the Riemann
tensor contained within the expression for $G_{1}$, these may
be integrated by parts so that
\begin{align}\label{eq-L2EstimateGradient2FF}
\begin{split}
\int_{\Sigma}c^{-1}\phi\psi^{2}|\overline{\nabla}A|^{2} \leq&
\int_{\Sigma}\phi|A|^{2}|\overline{\nabla}\psi|^{2}\\
& +\int_{\Sigma}\phi\psi^{2}(|Riem|^{2}+|\overline{\nabla}H|^{2}
+|\nabla\log\phi|^{2}|Riem|^{2})\\
& +\int_{\Sigma}\phi\psi^{2}|A|^{2}(1+|Riem|+|\nabla\log\phi|^{2}
+|N(F)|+(1+|\nabla\log\phi|)|F|)\\
& +\int_{\Sigma}\phi\psi^{2}|A|^{3}(|H|+|\nabla\log\phi|).
\end{split}
\end{align}

Now observe that according to \eqref{eq-Integral2FF} with $p=1$,
\begin{align*}
\int_{\Sigma}\phi\psi^{2}|A|^{4} &\leq \int_{\Sigma}\phi|\overline{\nabla}(|A|\psi)|^{2}
+\frac{1}{2}\phi\psi^{2}G_{2}|A|^{2}\\
&\leq \int_{\Sigma}2\phi|A|^{2}|\overline{\nabla}\psi|^{2}
+\phi\psi^{2}\left(2|\overline{\nabla}|A||^{2}+\frac{1}{2}G_{2}|A|^{2}\right).
\end{align*}
Also
\begin{align*}
|\overline{\nabla}|A||^{2} &= \overline{g}^{ij}\partial_{i}|A|\partial_{j}|A|\\
&= \overline{g}^{ij}(|A|^{-1}A^{ab}\overline{\nabla}_{i}A_{ab})
(|A|^{-1}A^{cd}\overline{\nabla}_{j}A_{cd})\\
&\leq c|\overline{\nabla}A|^{2}.
\end{align*}
Therefore
\begin{equation*}
\int_{\Sigma}\phi\psi^{2}|A|^{4}
\leq\int_{\Sigma}2\phi|A|^{2}|\overline{\nabla}\psi|^{2}
+\frac{1}{2}\phi\psi^{2}G_{2}|A|^{2}+c\phi\psi^{2}|\overline{\nabla}A|^{2}.
\end{equation*}
Combining this with \eqref{eq-L2EstimateGradient2FF} then yields
\begin{align*}
\int_{\Sigma}c^{-1}\phi\psi^{2}|A|^{4} \leq &
\int_{\Sigma}\phi|A|^{2}|\overline{\nabla}\psi|^{2}\\
& +\int_{\Sigma}\phi\psi^{2}(|Riem|^{2}
+|\overline{\nabla}H|^{2}+|\nabla\log\phi|^{2}|Riem|^{2})\\
& +\int_{\Sigma}\phi\psi^{2}|A|^{2}
(1+|Riem|+|\nabla\log\phi|^{2}+|N(F)|+(1+|\nabla\log\phi|)|F|)\\
& +\int_{\Sigma}\phi\psi^{2}|A|^{3}(|H|+|\nabla\log\phi|).
\end{align*}
Now replace $\psi$ with $\psi^{2}$ and apply the Schwarz inequality to obtain
\begin{equation}\label{PrelL4Estimate}
\int_{\Sigma}c^{-1}\phi\psi^{4}|A|^{4}
\leq \int_{\Sigma}\phi|\overline{\nabla}\psi|^{4}
+\phi\psi^{4}(1+|H|^{4}+|\overline{\nabla}H|^{2}
+|Riem|^{4}+|\nabla\log\phi|^{4}+|F|^{4}+|N(F)|^{2}).
\end{equation}

%Our next task is to turn this integral estimate for
%$|A|$ into a pointwise estimate by using a Moser iteration.
Let $B_{r_{0}}^{4}(x_{0})$
be a geodesic ball in $(M\times\mathbb{R},g+\phi^{2}dt^{2})$
centered at a point $x_{0}\in\Sigma$, with $r_{0}$ less than
the injectivity
radius. Let $r$ be the distance function in $M\times\mathbb{R}$
and choose the cut-off function $\psi$ to be a function of $r$ such that
$\psi(r)=1$ for $r\leq r_{0}/2$ and $\psi(r)=0$ for $r\geq r_{0}$.
Since $|\nabla r|=1$ and hence $|\overline{\nabla}r|\leq 1$, we may further
choose $\psi$ so that $|\psi|\leq 1$ and
$|\overline{\nabla}\psi|\leq 3 r_{0}^{-1}$. In order to
bound the volume of $\Sigma\cap B_{r_{0}}^{4}(x_{0})$, note that by
\eqref{eq-EstimateMeanCurv}
\begin{equation*}
|div_{M\times\mathbb{R}}(N)| = |H| \leq |F|+c(1+|\nabla\log\phi|).
\end{equation*}
Upon an integration by parts
over the region $B_{r_{0}}^{4}(x_{0})\cap\{(x,t)\mid t<f(x)\}$,
we  obtain
\begin{equation}\label{eq-VolumeEstimate}
Vol(\Sigma\cap B_{r_{0}}^{4}(x_{0}))\leq cr_{0}^{3}
\end{equation}
where $c$ depends on $|F|$ and $|\nabla\log\phi|$.
Applying this result together with \eqref{PrelL4Estimate} then yields
\begin{equation}\label{eq-L4Estimates2FF}
\int_{\Sigma\cap B_{\frac{r_{0}}{2}}^{4}(x_{0})}|A|^{4}\leq c,
\end{equation}
where $c$ depends on $|F|$, $|N(F)|$, $|H|$, $|\overline{\nabla}H|$,
$|\nabla\log\phi|$, $|Riem|$, $\frac{\max\phi}{\min\phi}$, and $r_{0}^{-1}$.
This is the desired $L^4$-estimate for the second fundamental form.

{\it Step 3. A pointwise estimate of the second fundamental form.}
Set $u=|A|^{2}+1.$ Then \eqref{eq-SimonIdentity} yields
\begin{equation*}
\overline{\Delta}u\geq -c(1+|A|^{2}+|F|^{2}+|\nabla\log\phi|^{2})u
+2|\overline{\nabla}A|^{2}+2A^{ij}(\overline{\nabla}_{ij}H+G_{ij}).
\end{equation*}
Multiply both sides by a nonnegative function
$\xi\in C^{\infty}_{c}(\Sigma\cap B_{\frac{r_{0}}{2}}^{4}(x_{0}))$
and integrate by parts to find
\begin{align*}
0 \geq& \int_{\Sigma}\langle\overline{\nabla}\xi,\overline{\nabla}u\rangle
-c(1+|A|^{2}+|F|^{2}+|\nabla\log\phi|^{2})\xi u-2A^{ij}\overline{\nabla}_{i}\xi\overline{\nabla}_{j}H\\
& +\int_{\Sigma}2\xi(|\overline{\nabla}A|^{2}+A^{ij}G_{ij}
-\overline{\nabla}_{i}A^{ij}\overline{\nabla}_{j}H).
\end{align*}
Although $G_{ij}$ contains first derivatives of the Riemann tensor,
these derivatives may be integrated by parts. The first derivatives
of $A$ which result from this process may be absorbed into
$|\overline{\nabla}A|^{2}$ to yield a more simple expression
\begin{equation*}
0\geq\int_{\Sigma}\overline{\nabla}\xi\cdot
\overline{\nabla}u+D^{i}(\overline{\nabla}_{i}\xi)u+D\xi u,
\end{equation*}
where
\begin{equation*}
D^{i}=-2u^{-1}A^{ij}\overline{\nabla}_{j}H-2u^{-1}A^{ij}\overline{g}^{nm}R_{Nnjm}
-2u^{-1}A^{nm}\overline{g}^{ij}R_{Nnmj},
\end{equation*}
and
\begin{equation*}
D=-c(1+|A|^{2}+|F|^{2}+|\nabla\log\phi|^{2}
+u^{-1}|\overline{\nabla}H|^{2}+u^{-1}|A|^{2}|Riem|+u^{-1}|Riem|^{2}).
\end{equation*}
By \eqref{eq-EstimateGradientMeanCurv}, we have
\begin{equation}\label{eq-CoefficientBound}
\sup_{\Sigma\cap B_{\frac{r_{0}}{2}}^{4}(x_{0})}|D^i|^{2}+\int_{\Sigma\cap B_{\frac{r_{0}}{2}}^{4}(x_{0})}D^{2}\leq c.
\end{equation}

We now show that the Sobolev inequality holds on
$\Sigma\cap B_{r_{0}}^{4}(x_{0})$. The volume estimate
\eqref{eq-VolumeEstimate} allows
an application of the Hoffman and Spruck inequality \cite{HoffmanSpruck} for a sufficiently small $r_{0}$. In particular, for
any function $\psi\in C^{\infty}_{c}(\Sigma\cap B_{r_{0}}^{4}(x_{0}))$, we have
\begin{equation*}
\left(\int_{\Sigma}\psi^{6}\right)^{1/3}\leq c\int_{\Sigma}(|\overline{\nabla}\psi|^{2}+H^{2}\psi^{2}).
\end{equation*}
Moreover, by \eqref{eq-EstimateMeanCurv} and Holder's inequality
\begin{equation*}
c^{-1}\left(\int_{\Sigma}\psi^{6}\right)^{1/3}\leq
\int_{\Sigma}|\overline{\nabla}\psi|^{2}
+r_{0}^{2}(1+|F|^{2}+|\nabla\log\phi|^{2})\left(\int_{\Sigma}\psi^{6}\right)^{1/3}.
\end{equation*}
Thus, if $r_{0}$ is sufficiently small we obtain the desired Sobolev inequality
\begin{equation}\label{eq-SobolevIneq}
\left(\int_{\Sigma}\psi^{6}\right)^{1/3}\leq c\int_{\Sigma}
|\overline{\nabla}\psi|^{2}.
\end{equation}

A Moser iteration \cite{GilbargTrudinger} can now be applied
with the help of \eqref{eq-CoefficientBound} and \eqref{eq-SobolevIneq}, since
$2>\frac{1}{2}\dim\Sigma=\frac{3}{2}$. It follows that
\begin{equation*}
u(x_{0})\leq c
\left(\int_{\Sigma\cap B_{\frac{r_{0}}{2}}^{4}(x_{0})}u^{2}\right)^{1/2}.
\end{equation*}
Now the desired estimate \eqref{eq-PointwiseEstimate2FF} follows from the
definition of $u$ and \eqref{eq-L4Estimates2FF}.
\end{proof}

As a corollary of Theorem \ref{thm-PointwiseEsti2FF}, we have
a pointwise gradient estimate for the mean curvature, which follows easily from
\eqref{eq-EstimateGradientMeanCurv} and \eqref{eq-PointwiseEstimate2FF}.

\begin{cor}\label{thm-PointwiseEstiGradientMeanCurv}
Suppose that the surface $\Sigma$ satisfies \eqref{eq-BasicEquation}.
Then
\begin{equation}\label{eq-PointwiseEstimateGradientMean}
\sup_{\Sigma}|\overline{\nabla}H|\leq c
\end{equation}
where $c$ depends on $|\log\phi|$, $|\nabla\log\phi|$,
$|\nabla^{2}\log\phi|$, $|F|$, and $|N(F)|$.
\end{cor}

Let $y^1$, $y^2$, $y^3$, $y^4$ be normal coordinates for the
warped product metric near $x_{0}\in\Sigma$, that is
\begin{equation*}
\widehat{g}=g+\phi^{2}dt^2=\widehat{g}_{ab}dy^a dy^b,
\text{ }\text{ }\text{ }\text{ }\widehat{g}_{ab}(0)=\delta_{ab},
\text{ }\text{ }\text{ }\text{ }
\partial_{y^c}\widehat{g}_{ab}(0)=0.
\end{equation*}
Suppose that $\partial_{y^{i}}\in T_{x_{0}}\Sigma$, $i=1,2,3$
so that near $x_{0}$, $\Sigma$ is given by a graph $y^4=w(y)$,
$y=(y^1,y^2,y^3)$. Equation \eqref{eq-BasicEquation} can now be written as
\begin{equation}\label{eq-ParametricForm}
\sum_{a,b=1}^{4}\left(\widehat{g}^{ab}-\frac{W^{a}W^{b}}{|\nabla W|^{2}}\right)\left(\frac{\nabla_{a}\nabla_{b}W}{|\nabla W|}-K_{ab}\right)=F,
\end{equation}
where $W(Y)=w(y)-y^4$, $Y=(y^1,y^2,y^3,y^4)$. To see this first observe that $Y_{i}=\partial_{y^{i}}+w_{,i}\partial_{y^{4}}$, $i=1,2,3$, form a basis for
the tangent space to $\Sigma$. Let $N=N_{a}dy^{a}$ be the unit normal
written as a 1-form. Then
\begin{equation*}
0=\widehat{g}(Y_{i},N)=Y_{i}^{a}N_{a}=\sum_{a=1}^{3}\delta_{i}^{a}N_{a}
+w_{,i}N_{4}=N_{i}+w_{,i}N_{4}.
\end{equation*}
It follows that
\begin{equation*}
N=\frac{\sum_{a=1}^{3}w_{,a}dy^{a}-dy^{4}}{|\nabla W|}:=\frac{\widetilde{N}}{|\widetilde{N}|}.
\end{equation*}
The second fundamental form, in these coordinates, is then given by
\begin{equation*}
A_{ij}=\widehat{g}(\nabla_{Y_{i}}N,Y_{j})=
\frac{\nabla_{i}\widetilde{N}_{j}}{|\widetilde{N}|}=
\frac{\nabla_{i}\nabla_{j}W}{|\nabla W|}.
\end{equation*}
Moreover, the induced metric on $\Sigma$ is $\widehat{g}_{ab}-N_{a}N_{b}$,
with inverse
\begin{equation*}
\widehat{g}^{ab}-N^{a}N^{b}=\widehat{g}^{ab}-\frac{W^{a}W^{b}}{|\nabla W|^{2}}.
\end{equation*}
This yields the desired expression for the generalized Jang equation.

Parametric estimates may now be obtained as another application of Theorem \ref{thm-PointwiseEsti2FF}.

\begin{theorem}\label{thm-ParametricEstimates}
Suppose that the surface $\Sigma$ satisfies \eqref{eq-BasicEquation}.
Then there exist constants $\rho>0$ and $c$ independent
of $\Sigma$, but dependent on $|\log\phi|_{C^{2,\alpha}}$
and $|F|_{C^{1,\alpha}}$, such that
$\Sigma\cap B^{4}_{\rho}(x_{0})\subset\{Y\mid y^{4}=w(y)\}$ and $|w|_{C^{3,\alpha}(B_{\rho})}\leq c$.
\end{theorem}

\begin{proof}
Observe that \eqref{eq-ParametricForm} is a strictly elliptic equation for $w$ having the following structure
\begin{equation*}
\sum_{i,j=1}^{3}B^{ij}(y,w,\partial w)w_{y^{i}y^{j}}=C(y,w,\partial w),
\end{equation*}
where $B^{ij}(0,0,0)=\delta^{ij}$. Since the second fundamental form
is uniformly bounded and
\begin{equation*}
|A|^{2}=\sum_{a,b,c,d=1}^{4}\left(\widehat{g}^{ac}-\frac{W^{a}W^{c}}{|\nabla W|^{2}}\right)\left(\widehat{g}^{bd}-\frac{W^{b}W^{d}}{|\nabla W|^{2}}\right)
\left(\frac{\nabla_{a}\nabla_{b}W}{|\nabla W|}\right)\left(\frac{\nabla_{c}\nabla_{d}W}{|\nabla W|}\right),
\end{equation*}
we obtain
\begin{equation*}
\sum_{i,j=1}^{3}(w_{y^{i}y^{j}})^{2}\leq c
\left(1+\sum_{i=1}^{3}(w_{y^{i}})^{2}\right)^{3},
\end{equation*}
near $y=0$. Now by some simple calculus,
\begin{equation*}
\sup_{|y|\leq\rho}(|w(y)|+|\partial w(y)|+|\partial^{2}w(y)|)\leq c.
\end{equation*}
Therefore the $C^{3,\alpha}$ estimates follow from Schauder's
theory and the $C^{2}$ estimates above.

Lastly, we note that $\Sigma$ may be expressed as a graph for sufficiently small
(but uniform) $\rho>0$.  This follows directly from the pointwise bound on the
second fundamental form established in Theorem \ref{thm-PointwiseEsti2FF}.
\end{proof}

\section{The Harnack Inequality}\label{sec3}

In this section it will be shown that the quantity
$\langle\partial_{t},N\rangle$ satisfies a homogeneous elliptic
equation with bounded coefficients, when $|A|$ is pointwise bounded.
From this we immediately obtain the Harnack
inequality as in \cite{SchoenYau}, which is used extensively.
The proof here consists of a long calculation.
Although the resulting equation has the same structure as that of
\cite{SchoenYau}, we cannot simply cite this
reference, as the calculations carried out here must be done
in the warped product setting.

We begin by recalling that the vector fields
$X_{i}=\partial_{i}+f_{i}\partial_{t}$ form a basis for the tangent
space to the surface $\Sigma=\{t=f(x)\}$, and that the vector field
\begin{equation*}
N=\frac{f^{m}\partial_{m}-\phi^{-2}\partial_{t}}{\sqrt{\phi^{-2}
+|\nabla f|^{2}}}
\end{equation*}
is its unit normal. Then
\begin{equation*}
\langle\partial_{t},N\rangle=\frac{-1}{\sqrt{\phi^{-2}+|\nabla f|^{2}}}.
\end{equation*}
Our goal will be to calculate the Laplacian
\begin{equation*}
\overline{\Delta}\langle\partial_{t},N\rangle=
\overline{g}^{pj}\overline{\nabla}_{X_{p}}\overline{\nabla}_{X_{j}}
\langle\partial_{t},N\rangle
\end{equation*}
in terms of $\langle\partial_{t},N\rangle$.
As usual, in this notation, an over line bar indicates
that the particular geometric quantity is with respect to
the induced metric on $\Sigma$.

The following lemma may be obtained from the Jacobi equation, as indicated in \cite{AnderssonEichmairMetzger}.
However here, we carry out a different (and longer) proof, as some of the calculations can be used later.

\begin{lemma}\label{lemma-LaplacianInnerProduct}
For any $\Sigma\subset(M\times\mathbb{R}, g+\phi^{2}dt^{2})$, not necessarily satisfying \eqref{eq-BasicEquation},
\begin{equation}\label{eq-LaplacianInnerProdct}
\overline{\Delta}\langle\partial_{t},N\rangle
+(|A|^{2}+N(H)+R_{NN})\langle\partial_{t},N\rangle=0.
\end{equation}
\end{lemma}

\begin{proof}
First observe that
\begin{equation*}
A_{ij}=\langle X_{i},\nabla_{X_{j}}N\rangle,
\end{equation*}
so that
\begin{equation*}
\nabla_{X_{j}}N=A_{jp}\overline{g}^{pa}X_{a}.
\end{equation*}
It follows that
\begin{equation*}
\overline{\nabla}_{X_{j}}\langle\partial_{t},N\rangle=
\langle\nabla_{X_{j}}\partial_{t},N\rangle
+\langle\partial_{t},X_{m}\rangle\overline{g}^{mn}A_{jn},
\end{equation*}
and
\begin{equation*}
\overline{\nabla}_{X_{p}}\overline{\nabla}_{X_{j}}\langle\partial_{t},N\rangle
=\overline{\nabla}_{X_{p}}\langle\nabla_{X_{j}}\partial_{t},N\rangle
+\overline{g}^{mn}(\overline{\nabla}_{p}A_{jn})\langle\partial_{t},X_{m}\rangle
+\overline{g}^{mn}A_{jn}\overline{\nabla}_{X_{p}}\langle\partial_{t},X_{m}\rangle.
\end{equation*}
The Codazzi equations \eqref{eq-Codazzi} yield
\begin{equation*}
\overline{\nabla}_{p}A_{jn}=\overline{\nabla}_{n}A_{jp}+R_{NX_{j}X_{n}X_{p}}.
\end{equation*}
Moreover
\begin{align*}
\overline{\nabla}_{X_{p}}\langle\partial_{t},X_{m}\rangle &= X_{p}\langle\partial_{t},X_{m}\rangle
-\overline{\Gamma}_{pm}^{a}\langle\partial_{t},X_{a}\rangle\\
&= \langle\nabla_{X_{p}}\partial_{t},X_{m}\rangle
+\langle\partial_{t},\nabla_{X_{p}}X_{m}\rangle
-\overline{\Gamma}_{pm}^{a}\langle\partial_{t},X_{a}\rangle,
\end{align*}
and
\begin{equation*}
\nabla_{X_{p}}X_{m}=\overline{\Gamma}_{pm}^{a}X_{a}-A_{pm}N.
\end{equation*}
Thus
\begin{align}\label{eq-InnerProduct}
\begin{split}
\overline{\Delta}\langle\partial_{t},N\rangle =& \overline{g}^{pj}
\overline{\nabla}_{X_{p}}
\overline{\nabla}_{X_{j}}\langle\partial_{t},N\rangle\\
=& \overline{g}^{pj}\overline{\nabla}_{X_{p}}
\langle\nabla_{X_{j}}\partial_{t},N\rangle-|A|^{2}\langle\partial_{t},N\rangle\\
& +\overline{g}^{mn}(\overline{\nabla}_{n}H+R_{NX_{n}})
\langle\partial_{t},X_{m}\rangle
+A^{pm}\langle\nabla_{X_{p}}\partial_{t},X_{m}\rangle.
\end{split}
\end{align}
The right-hand side will be calculated term by term.

We first claim that
\begin{equation}\label{eq-Claim1}
\overline{g}^{mn}(\overline{\nabla}_{n}H+R_{NX_{n}})
\langle\partial_{t},X_{m}\rangle
=-\langle\partial_{t},N\rangle(\phi^{-1}\Delta\phi+N(H)+R_{NN}).
\end{equation}
To prove this, write
\begin{equation*}
\partial_{t}=\langle\partial_{t},N\rangle N
+\overline{g}^{mn}\langle\partial_{t},X_{m}\rangle X_{n}
\end{equation*}
and observe that
\begin{equation*}
\overline{g}^{mn}(\overline{\nabla}_{n}H+R_{NX_{n}})
\langle\partial_{t},X_{m}\rangle
=Ric(N,\partial_{t})+\partial_{t}H-\langle\partial_{t},N\rangle(R_{NN}+N(H)).
\end{equation*}
The desired result follows since $\partial_{t}H=0$ and
\begin{equation*}
Ric(N,\partial_{t})=-\langle\partial_{t},N\rangle\phi^{-1}\Delta\phi.
\end{equation*}
To see this last assertion, use the fact that
\begin{equation}\label{eq-ChristofolSymbols}
\Gamma_{44}^{4}=\Gamma_{ij}^{4}=\Gamma_{i4}^{j}=0,\text{ }\text{ }
\text{ }\text{ }\Gamma_{i4}^{4}=(\log\phi)_{i},\text{ }\text{ }\text{ }\text{ }
\Gamma_{44}^{i}=-\phi\phi^{i},
\end{equation}
to calculate
\begin{equation*}
R_{4ijk}=0,\text{ }\text{ }\text{ }\text{ }R_{4i4j}=-\phi\nabla_{ij}\phi.
\end{equation*}
Here 4 indicates the $t$-coordinate. This implies that
\begin{equation*}
R_{4i}=0,\text{ }\text{ }\text{ }\text{ }R_{44}=-\phi\Delta\phi,
\end{equation*}
and hence
\begin{equation*}
Ric(N,\partial_{t})=-\langle\partial_{t},N\rangle Ric(f^{m}\partial_{m}-\phi^{-2}\partial_{t},\partial_{t})
=-\langle\partial_{t},N\rangle \phi^{-1}\Delta\phi.
\end{equation*}
This finishes the proof of \eqref{eq-Claim1}.

Next, with the help of \eqref{eq-ChristofolSymbols},
a straightforward calculation yields
\begin{equation*}
\nabla_{X_{p}}\partial_{t}=(\log\phi)_{p}\partial_{t}-\phi f_{p}\phi^{m}\partial_{m}.
\end{equation*}
Therefore
\begin{equation*}
\langle\nabla_{X_{p}}\partial_{t},X_{m}\rangle=\phi(f_{m}\phi_{p}-f_{p}\phi_{m}).
\end{equation*}
Since this is an antisymmetric tensor, and $A^{pm}$ is symmetric, we obtain
\begin{equation}\label{eq-Claim2}
A^{pm}\langle\nabla_{X_{p}}\partial_{t},X_{m}\rangle=0.
\end{equation}
Alternatively, since $\partial_{t}$ is a Killing field, it follows immediately that
$\langle\nabla_{X_{p}}\partial_{t},X_{m}\rangle$ is antisymmetric.

Lastly, we claim that
\begin{equation}\label{eq-Claim3}
\overline{g}^{pj}\overline{\nabla}_{X_{p}}
\langle\nabla_{X_{j}}\partial_{t},N\rangle=
\langle\partial_{t},N\rangle \phi^{-1}\Delta\phi.
\end{equation}
First observe that since $\partial_{t}$ is a Killing field,
\begin{equation*}
\langle\nabla_{X_{j}}\partial_{t},N\rangle=
-\langle\nabla_{N}\partial_{t},X_{j}\rangle.
\end{equation*}
Furthermore
\begin{equation*}
\nabla_{N}\partial_{t}=\nabla_{\partial_{t}}N
-L_{\partial_{t}}N=\nabla_{\partial_{t}}N,
\end{equation*}
where $L$ denotes Lie differentiation. It follows that
\begin{equation*}
\langle\nabla_{X_{j}}\partial_{t},N\rangle=
-\langle\nabla_{\partial_{t}}N,X_{j}\rangle.
\end{equation*}
Now calculate
\begin{align*}
\nabla_{\partial_{t}}N &= (\phi^{-2}+|\nabla f|^{2})^{-1/2}
(f^{i}\Gamma_{i4}^{a}\partial_{a}
-\phi^{-2}\Gamma_{44}^{a}\partial_{a})\\
&= -\langle\partial_{t},N
\rangle(f^{i}(\log\phi)_{i}\partial_{t}+\phi^{-1}\phi^{i}\partial_{i})\\
&= -\langle\partial_{t},N\rangle(\log\phi)^{i}X_{i},
\end{align*}
to find
\begin{equation*}
\overline{g}^{pj}\langle\nabla_{\partial_{t}}N,X_{j}\rangle
=-\langle\partial_{t},N\rangle(\log\phi)^{p}.
\end{equation*}
We then have
\begin{align*}
\overline{\nabla}_{X_{p}}(\overline{g}^{pj}
\langle\nabla_{\partial_{t}}N,X_{j}\rangle)
&= X_{p}(\overline{g}^{pj}\langle\nabla_{\partial_{t}}N,X_{j}\rangle)
+\overline{\Gamma}_{pn}^{p}\overline{g}^{nj}
\langle\nabla_{\partial_{t}}N,X_{j}\rangle\\
&= -(\partial_{p}\langle\partial_{t},N\rangle)(\log\phi)^{p}
-\langle\partial_{t},N\rangle\partial_{p}(\log\phi)^{p}
-\langle\partial_{t},N\rangle\overline{\Gamma}_{pn}^{p}(\log\phi)^{n}\\
&= -\langle\partial_{t},N\rangle\Delta(\log\phi)
+\langle\partial_{t},N\rangle(\Gamma_{pn}^{p}
-\overline{\Gamma}_{pn}^{p})(\log\phi)^{n}
-(\partial_{p}\langle\partial_{t},N\rangle)(\log\phi)^{p}.
\end{align*}
According to a calculation in \cite{BrayKhuri} (page 760),
\begin{equation*}
\Gamma_{pj}^{n}-\overline{\Gamma}_{pj}^{n}
=\phi\phi^{n}f_{p}f_{j}-\frac{f^{n}A_{pj}}{\sqrt{\phi^{-2}+|\nabla f|^{2}}}.
\end{equation*}
Moreover, by \eqref{eq-SecondFF},
\begin{align*}
-\partial_{p}\langle\partial_{t},N\rangle &= \partial_{p}(\phi^{-2}
+|\nabla f|^{2})^{-1/2}\\
&= \frac{\phi^{-3}\phi_{p}-f^{n}\nabla_{pn}f}{(\phi^{-2}
+|\nabla f|^{2})^{3/2}}\\
&= \frac{\phi^{-3}\phi_{p}}{(\phi^{-2}+|\nabla f|^{2})^{3/2}}
-\frac{f^{n}A_{pn}}{\phi^{-2}+|\nabla f|^{2}}
+\frac{f^{n}((\log\phi)_{p}f_{n}
+(\log\phi)_{n}f_{p}+\phi f^{m}\phi_{m}f_{p}f_{n})}
{(\phi^{-2}+|\nabla f|^{2})^{3/2}}.
\end{align*}
Therefore
\begin{align*}
& \langle\partial_{t},N\rangle(\Gamma_{pn}^{p}
-\overline{\Gamma}_{pn}^{p})(\log\phi)^{n}
-(\partial_{p}\langle\partial_{t},N\rangle)(\log\phi)^{p}\\
=& -\frac{(\log\phi)^{n}\phi\phi^{m}f_{m}f_{n}}{(\phi^{-2}+|\nabla f|^{2})^{1/2}}
+\frac{(\log\phi)^{p}\phi^{-3}\phi_{p}}{(\phi^{-2}+|\nabla f|^{2})^{3/2}}\\
& +\frac{(\log\phi)^{p}f^{n}[(\log\phi)_{p}f_{n}+(\log\phi)_{n}f_{p}+\phi f^{m}\phi_{m}f_{p}f_{n}]}
{(\phi^{-2}+|\nabla f|^{2})^{3/2}}\\
=& \frac{|\nabla\log\phi|^{2}}{(\phi^{-2}+|\nabla f|^{2})^{1/2}},
\end{align*}
and \eqref{eq-Claim3} follows.

The desired identity is now obtained by substituting \eqref{eq-Claim1}, \eqref{eq-Claim2}, and
\eqref{eq-Claim3}, into \eqref{eq-InnerProduct}.
\end{proof}

We are now ready to prove the Harnack inequality.

\begin{theorem}\label{thm-HarnackInequality}
Suppose that the surface $\Sigma$ satisfies \eqref{eq-BasicEquation}.
Then there exist constants $\rho>0$ and $c$ independent of $\Sigma$, but dependent on $|\log\phi|_{C^{2,\alpha}}$ and $|F|_{C^{1,\alpha}}$, such that
\begin{equation}\label{eq-HarnackIneq}
\sup_{\Sigma\cap B^{4}_{\rho}(x_{0})}
\langle\partial_{t},-N\rangle\leq c\inf_{\Sigma\cap B^{4}_{\rho}(x_{0})}\langle\partial_{t},-N\rangle
\end{equation}
and
\begin{equation}\label{eq-DifferentialHarnack}
\sup_{\Sigma\cap B^{4}_{\rho}(x_{0})}|\overline{\nabla}
\log\langle\partial_{t},-N\rangle|\leq c.
\end{equation}
\end{theorem}

\begin{proof} The estimates of Theorem \ref{thm-ParametricEstimates}
guarantee that equation \eqref{eq-LaplacianInnerProdct} is
uniformly elliptic for small $\rho$.
Next, we note that the coefficients of \eqref{eq-LaplacianInnerProdct}
are bounded. To see this, observe that $R_{NN}$ poses no problem, and
$|A|^2$ is bounded by Theorem \ref{thm-PointwiseEsti2FF}. Moreover
\begin{equation*}
N(H)=\frac{\phi f^{m}\partial_{m}H}{\sqrt{1+\phi^{2}|\nabla f|^{2}}},
\end{equation*}
so that $N(H)$ is bounded by Corollary \ref{thm-PointwiseEstiGradientMeanCurv}.

The Harnack inequality \eqref{eq-HarnackIneq} now follows
immediately.
Standard elliptic theory \cite{GilbargTrudinger} also guarantees that
\begin{equation*}
\sup_{B_{\frac{\rho}{2}}}|\overline{\nabla}\langle\partial_{t},N\rangle|\leq c\sup_{B_{\rho}}|\langle\partial_{t},N\rangle|.
\end{equation*}
Combining this with \eqref{eq-HarnackIneq} yields
\begin{equation*}
\sup_{B_{\frac{\rho}{2}}}|\overline{\nabla}\langle\partial_{t},N\rangle|\leq c\inf_{B_{\rho}}|\langle\partial_{t},N\rangle|\leq c\inf_{B_{\frac{\rho}{2}}}|\langle\partial_{t},N\rangle|,
\end{equation*}
from which \eqref{eq-DifferentialHarnack} follows.
\end{proof}

\section{Global $C^{1}$ Bounds and Existence for the Regularized Equation}\label{sec4}

Suppose that the boundary of $M$ is an outermost future apparent
horizon with one component, that is
\begin{equation}\label{131}
\theta_{+}(\partial M)=0
\end{equation}
where the null expansion $\theta_{+}$ was defined in \eqref{9}.
All arguments to follow may be easily extended to
the general case in which the boundary is an outermost apparent
horizon, with several future and past horizon components.
Take a sufficiently large coordinate sphere $\partial_{\infty}M$
in the asymptotically flat end such that
\begin{equation}\label{132}
\theta_{+}(\partial_{\infty}M)>0.
\end{equation}
Following \cite{AnderssonMetzger} and \cite{Metzger} we construct
an extension $(\widetilde{M},\widetilde{g},\widetilde{k})$, of the
initial data $(M,g,k)$, having
the following properties:

\noindent i) $M\subset\widetilde{M}$ with $\widetilde{g}|_{M}=g$,
$\widetilde{k}|_{M}=k$, and $\partial_{\infty}\widetilde{M}=
\partial_{\infty}M$,

\noindent ii) $\widetilde{g}$ is smooth across $\partial M$ and
$\widetilde{k}$ is $C^{1,1}$ across $\partial M$,

\noindent iii) the region $\widetilde{M}-M$ is foliated by
surfaces $S_{\sigma}$, $\sigma\in[0,\sigma_{0})$, with $\theta_{+}(S_{\sigma})<0$, $S_{0}=\partial\widetilde{M}$,
and $S_{\sigma_{0}}=\partial M$,

\noindent iv) for all sufficiently small $\sigma$,
$\theta_{+}(S_{\sigma})\leq-\vartheta<0$ and $\widetilde{k}\equiv 0$.

For convenience, in what follows, we will denote $\widetilde{g}$
by $g$ and $\widetilde{k}$ by $k$.
Moreover, we will extend the warping factor $\phi$ to be positive
on all of $\widetilde{M}$.
These constructions allow for the existence of appropriate barriers,
which in turn leads to a solution for the following Dirichlet
problem
\begin{equation}\label{133}
H(f_{\varepsilon})-K(f_{\varepsilon})=\varepsilon f_{\varepsilon}
\text{ }\text{ }\text{ on }\text{ }\text{ }\widetilde{M},
\end{equation}
\begin{equation}\label{134}
f_{\varepsilon}=\frac{\vartheta}{2\varepsilon}\text{ }\text{ }
\text{ on }\text{ }\text{ }\partial\widetilde{M},\text{ }
\text{ }\text{ }\text{ }
f_{\varepsilon}=0\text{ }\text{ }\text{ on }
\text{ }\text{ }\partial_{\infty}\widetilde{M}.
\end{equation}

The first step is to obtain global $C^{1}$ bounds. Observe
that a direct application of the maximum principle yields the $C^{0}$ bound
\begin{equation}\label{135}
\sup_{\widetilde{M}}|f_{\varepsilon}|\leq c\varepsilon^{-1},
\end{equation}
where the constant $c$ depends on $|k|_{C^{0}}$ and $\vartheta$.
In order to apply the maximum principle to obtain bounds on
first derivatives,
we need to establish the boundary gradient estimates. Let $\tau(x)=dist(x,\partial\widetilde{M})$, and denote by $S_{\tau}$
the level sets of the
geodesic flow $\partial_{\tau}=n$ emanating from
$\partial\widetilde{M}$, where $n$ is the unit outer normal
(pointing towards spatial infinity) of the
surfaces $S_{\tau}$. The barrier functions $\psi_{\pm}(\tau)$
will be functions of $\tau$ alone. In order to find sub and
super solutions notice that
\begin{align*}\label{136}
\begin{split}
& \left(g^{ij}-
\frac{\phi^{2}\psi^{i}\psi^{j}}{1+\phi^{2}|\nabla\psi|^{2}}\right)
\left(\frac{\phi\nabla_{ij}\psi+\phi_{i}\psi_{j}+\phi_{j}\psi_{i}}
{\sqrt{1+\phi^{2}|\nabla\psi|^{2}}}\right)\\
=&\left(1-\frac{\phi^{2}\psi'^{2}}{1+\phi^{2}\psi'^{2}}\right)
\left(\frac{\phi\psi''+2\phi'\psi'}{\sqrt{1+\phi^{2}\psi'^{2}}}\right)
+(g|_{S_{\tau}})^{ij}\frac{\phi\nabla_{ij}\psi}{\sqrt{1+\phi^{2}\psi'^{2}}},
\end{split}
\end{align*}
and
\begin{equation*}\label{137}
\left(g^{ij}
-\frac{\phi^{2}\psi^{i}\psi^{j}}{1+\phi^{2}|\nabla\psi|^{2}}\right)k_{ij}=
\left(1-\frac{\phi^{2}\psi'^{2}}{1+\phi^{2}\psi'^{2}}\right)k_{nn}
+Tr_{S_{\tau}}k,
\end{equation*}
where $\psi'=\frac{d\psi}{d\tau}$. Moreover
\begin{equation*}\label{138}
\nabla_{ij}\psi=\partial_{ij}\psi-\Gamma_{ij}^{m}\partial_{m}\psi=
\nabla_{ij}^{S_{\tau}}\psi-\Gamma_{ij}^{n}\partial_{n}\psi
=\nabla_{ij}^{S_{\tau}}\psi+A^{S_{\tau}}_{ij}\partial_{n}\psi,
\end{equation*}
where $\nabla^{S_{\tau}}$ denotes the induced connection on
$S_{\tau}$ and $A^{S_{\tau}}$ is its second fundamental form.
Therefore
\begin{align}\label{139}
\begin{split}
& H(\psi)-K(\psi)-\varepsilon\psi\\
=& \frac{\phi\psi''+2\phi'\psi'}{(1+\phi^{2}\psi'^{2})^{3/2}}
+\frac{\phi\psi'H_{S_{\tau}}}{\sqrt{1+\phi^{2}\psi'^{2}}}
-\frac{k_{nn}}{1+\phi^{2}\psi'^{2}}-Tr_{S_{\tau}}k-\varepsilon\psi\\
=& \frac{\phi\psi''+2\phi'\psi'}{(1+\phi^{2}\psi'^{2})^{3/2}}
+\frac{\phi\psi'\theta_{+}(S_{\tau})}{\sqrt{1+\phi^{2}\psi'^{2}}}
-\frac{k_{nn}}{1+\phi^{2}\psi'^{2}}
-\left(1+\frac{\phi\psi'}{\sqrt{1+\phi^{2}\psi'^{2}}}\right)
Tr_{S_{\tau}}k-\varepsilon\psi.
\end{split}
\end{align}
Let $\psi_{-}(\tau)=a-b\tau$ where $a$ and $b$ will be chosen
appropriately and positive. In order that $\psi_{-}$ agree with
$f_{\varepsilon}$
at $\partial\widetilde{M}$ we set $a=\frac{\vartheta}{2\varepsilon}$.
Then for $\tau\in[0,\tau_{0}]$, with $\tau_{0}$ sufficiently small,
\begin{equation*}\label{140}
H(\psi_{-})-K(\psi_{-})
-\varepsilon\psi_{-}\geq \vartheta-\varepsilon a+\varepsilon b\tau+O(b^{-1})\geq\frac{\vartheta}{2}+O(b^{-1})\geq 0,
\end{equation*}
if $b$ is chosen large enough depending on $\phi$. Similarly an
upper barrier may be constructed in the form $\psi_{+}(\tau)=a+b\tau$.
Note that in the
case of an upper barrier we use the fact that $k\equiv 0$ near
$\partial\widetilde{M}$ to deal with the $Tr_{S_{\tau}}k$ term.
Now choose
$b$ so large that $a+b\tau_{0}\geq\sup_{\widetilde{M}}|f_{\varepsilon}|$
and $a-b\tau_{0}\leq-\sup_{\widetilde{M}}|f_{\varepsilon}|$, then by
a standard
comparison argument $\psi_{-}\leq f_{\varepsilon}\leq\psi_{+}$ for
$\tau\in[0,\tau_{0}]$. It follows that
\begin{equation}\label{141}
\sup_{\partial\widetilde{M}}|\nabla f_{\varepsilon}|\leq b\leq
c\varepsilon^{-1}.
\end{equation}
Moreover, a standard barrier construction at
$\partial_{\infty}\widetilde{M}$ also yields a boundary gradient estimate.

We now use the Bernstein method to obtain global $C^{1}$ estimates.
For convenience we temporarily drop the subscript $\varepsilon$ from
$f_{\varepsilon}$.
Differentiate equation \eqref{133} with respect to $\partial_{p}$ to find
\begin{align}\label{142}
\begin{split}
&\, \overline{g}^{ij}\left[\frac{\phi_{p}f_{;ij}+\phi f_{;ijp}+\phi_{;ip}f_{j}+\phi_{i}f_{;jp}+\phi_{;jp}f_{i}+\phi_{j}f_{;ip}}
{\sqrt{1+\phi^{2}|\nabla f|^{2}}}\right]\\
&\,\, -\overline{g}^{ij}\left[\frac{(\phi f_{;ij}+\phi_{i}f_{j}+\phi_{j}f_{i})(\phi\phi_{p}|\nabla f|^{2}
+\phi^{2}f_{;mp}f^{m})}
{(1+\phi^{2}|\nabla f|^{2})^{3/2}}-k_{ij;p}\right]\\
& \,\,-2\left[\frac{\phi f_{;ij}+\phi_{i}f_{j}+\phi_{j}f_{i}}{\sqrt{1+\phi^{2}|\nabla f|^{2}}}-k_{ij}\right]
\left[\frac{\phi\phi_{p}f^{i}f^{j}+\phi^{2}f_{;p}^{\!\text{ }\text{ }i}f^{j}}{1+\phi^{2}|\nabla f|^{2}}
-\frac{\phi^{2}f^{i}f^{j}(\phi\phi_{p}|\nabla f|^{2}
+\phi^{2}f_{;mp}f^{m})}{(1+\phi^{2}|\nabla f|^{2})^{2}}\right]\\
=&\,\varepsilon f_{p}.
\end{split}
\end{align}
Notice that the Ricci commutation formula yields
\begin{align*}\label{143}
\begin{split}
\overline{g}^{ij}\phi f^{p}f_{;ijp} &= \overline{g}^{ij}\phi f^{p}
(f_{;pij}+R^{m}_{\text{ }\text{ }ijp}f_{m})\\
&= \nabla_{j}(\overline{g}^{ij}\phi f^{p}f_{;pi})
-\overline{g}^{ij}\phi f_{;i}^{\!\text{ }\text{ }p}f_{;pj}-\nabla_{j}(\overline{g}^{ij}\phi)f^{p}f_{;pi}
+\overline{g}^{ij}\phi f^{p}f_{m}R^{m}_{\text{ }\text{ }ijp}.
\end{split}
\end{align*}
Thus, if $u=|\nabla f|^{2}$, then \eqref{142} implies that
\begin{equation*}\label{144}
\frac{\nabla _{j}(\phi\overline{g}^{ij} u_{i})}{2\sqrt{1+\phi^{2}|\nabla f|^{2}}}+B^{i}u_{i}+Bu^{1/2}
+\frac{\overline{g}^{ij}(f^{p}\phi_{p}f_{;ij}-\phi f_{;i}^{\!\text{ }\text{ }p}f_{;pj})}{\sqrt{1+\phi^{2}|\nabla f|^{2}}}
\geq\varepsilon u,
\end{equation*}
for some coefficients $B^{i}$ and $B$, with $B$ bounded.
Moreover at a critical point for $u$,
\begin{equation*}\label{145}
\overline{g}^{ij}(f^{p}\phi_{p}f_{;ij}-\phi f_{;i}^{\!\text{ }\text{ }p}f_{;pj})=g^{ij}(f^{p}\phi_{p}f_{;ij}
-\phi f_{;i}^{\!\text{ }\text{ }p}f_{;pj})
\leq-\frac{1}{2}\phi|\nabla^{2}f|^{2}+cu.
\end{equation*}
An alternate approach to deal with the $\overline{g}^{ij}f_{;ij}$
term is to solve for it in \eqref{133}, leaving
a manageable number of
first derivatives of $f$. Hence from the maximum principle we obtain
\begin{equation}\label{146}
\sup_{\widetilde{M}}|\nabla f_{\varepsilon}|\leq c\varepsilon^{-1}.
\end{equation}

With the global $C^{1}$ bounds the equation \eqref{133} is
uniformly elliptic. Standard theory \cite{GilbargTrudinger} then gives
$C^{1,\alpha}$ estimates up to the boundary. The Schauder estimates
may now be applied to obtain global $C^{2,\alpha}$ bounds. Thus we
may apply the continuity method, as in \cite{AnderssonMetzger},
to the family of equations
\begin{equation}\label{147}
H(f_{\varrho,\varepsilon})-\varrho K(f_{\varrho,\varepsilon})
=\varepsilon f_{\varrho,\varepsilon},\text{ }\text{ }\text{ }\text{ }0
\leq \varrho\leq 1,
\end{equation}
\begin{equation}\label{148}
f_{\varrho,\varepsilon}=\frac{\varrho\vartheta}{2\varepsilon}
\text{ }\text{ }\text{ on }\text{ }\text{ }\partial\widetilde{M},
\text{ }\text{ }\text{ }\text{ }
f_{\varrho,\varepsilon}=0\text{ }\text{ }\text{ on }\text{ }
\text{ }\partial_{\infty}\widetilde{M},
\end{equation}
to obtain a solution of boundary value problem \eqref{133},
\eqref{134}. Furthermore, by sending $\partial_{\infty}\widetilde{M}$ to infinity
we obtain a solution on all of $\widetilde{M}$ with the usual decay
\eqref{11}. Lastly, note that the global estimates of this section
also hold
when $M$ does not have boundary; of course they are easier to prove
in this case, as there is no need for an extension $\widetilde{M}$
or boundary gradient estimates.
Thus, in the case that $M$ does not have boundary, we may set
$F=\varepsilon f_{\varepsilon}$ and apply the results of Sections
\ref{sec2} and \ref{sec3},
as well as the global estimates of this section, and let
$\varepsilon\rightarrow 0$ to obtain the following result
(see \cite{SchoenYau} for details).

\begin{theorem}\label{thm1}
Suppose that $(M,g,k)$ is a smooth, complete, asymptotically flat
initial data set, and that $\phi$ is smooth, strictly positive,
and satisfies \eqref{10}.
Then there exist disjoint open sets $\Omega_{+},\Omega_{-}\subset M$,
and a smooth function $f:M-(\Omega_{+}\cup\Omega_{-})
\rightarrow\mathbb{R}$
satisfying
the generalized Jang equation \eqref{7} as well as \eqref{11}.
Furthermore $\partial\Omega_{+}$ ($\partial\Omega_{-}$) is a future
(past) apparent horizon
with $f(x)\rightarrow\pm\infty$ as $x\rightarrow\partial\Omega_{\pm}$.
More precisely $\mathrm{graph}(f)$ is asymptotic to the cylinders $\partial\Omega_{+}\times\mathbb{R}_{+}$
and $\partial\Omega_{-}\times\mathbb{R}_{-}$.
\end{theorem}

\section{Blow-Up Solutions for the Generalized Jang Equation}\label{sec5}

Consider the solutions $f_{\varepsilon}$ given in the previous section,
defined on $\widetilde{M}$ and with fall-off \eqref{11} at spatial
infinity.  Our goal in the current section is to produce solutions which blow-up at the
outermost apparent horizon boundary of $M$, by letting
$\varepsilon\rightarrow 0$. We will also
construct appropriate super solutions in order to obtain an estimate
for the rate of blow-up. As in the previous section we will assume here that
the boundary of $M$ is an outermost future apparent horizon with one component.
All arguments to follow may be easily extended to
the general case in which the boundary is an outermost apparent
horizon, with several future and past horizon components.

Observe that the gradient estimate
\eqref{146} and the boundary condition \eqref{134} imply
that there exists $\kappa$ independent of $\varepsilon$, such that
\begin{equation}\label{149}
f_{\varepsilon}(x)\geq\frac{\vartheta}{4\varepsilon}
\text{ }\text{ }\text{ for }\text{ }\text{ }
\mathrm{dist}(x,\partial\widetilde{M})<\kappa.
\end{equation}
As in \cite{SchoenYau} a subsequence of the $\varepsilon$-Jang
surfaces converges to a properly embedded submanifold
of $\widetilde{M}\times\mathbb{R}$. This convergence determines
three types of domains:
\begin{align}\label{150}
\begin{split}
\widetilde{M}_{+} =&\, \{x\in\widetilde{M}
\mid f_{\varepsilon_{i}}(x)\rightarrow+\infty\text{ }
\text{ locally uniformly as }\text{ }i\rightarrow\infty\},\\
\widetilde{M}_{-} =&\, \{x\in\widetilde{M}\mid f_{\varepsilon_{i}}(x)
\rightarrow-\infty\text{ }
\text{ locally uniformly as }\text{ }i\rightarrow\infty\},\\
\widetilde{M}_{0} =&\, \{x\in\widetilde{M}\mid
\limsup_{i\rightarrow\infty} |f_{\varepsilon_{i}}(x)|<\infty \}.
\end{split}
\end{align}
By \eqref{149}, $\widetilde{M}_{+}\neq \emptyset$ and $\widetilde{M}_{+}$
contains a neighborhood of $\partial\widetilde{M}$. Since
$\partial\widetilde{M}_{+}-\partial\widetilde{M}$ consists of apparent
horizons and the region $\widetilde{M}-M$ is foliated by surfaces
with $\theta_{+}<0$, we must have that $\widetilde{M}-M
\subset\widetilde{M}_{+}$. Thus $\partial\widetilde{M}_{+}$ is
an apparent horizon
in $M$. As $\partial M$ is the outermost apparent horizon in $M$,
we conclude that the closure of $\widetilde{M}_{+}$ is $\widetilde{M}-M$.
A standard barrier argument \cite{SchoenYau} at spatial infinity shows
that the $\varepsilon$-Jang surfaces are uniformly bounded there,
so that $\widetilde{M}_{0}$ contains a neighborhood of spatial
infinity. It follows that $\widetilde{M}_{0}=M$, and the limiting
Jang surface $\Sigma$ blows-up in the form of a cylinder over $\partial M$.
This line of argument was used by Metzger \cite{Metzger}
for the classical Jang equation.

This result holds for all warping factors $\phi$ which are strictly
positive up to and including the boundary. We would now like to examine
what happens when $\phi$ is allowed to vanish on the boundary.
From now on, let
$$\tau(x)=dist(x,\partial M),$$
the distance to the boundary, and denote by $S_{\tau}$
the level sets of the geodesic flow $\partial_{\tau}=n$
emanating from $\partial M$, where $n$ is the unit outer normal
(pointing towards spatial infinity) of the surfaces $S_{\tau}$.
We assume that
$$\phi=\tau^{b}\widetilde{\phi}\text{ }\text{ }\text{ near }\text{ }\text{ }\partial M,$$
where $\widetilde{\phi}$ is some
positive function, and then set
$$\phi_{\delta}=(\tau+\delta)^{b}\widetilde{\phi}
\text{ }\text{ }\text{ near }\text{ }\text{ }\partial M,$$
where $\delta>0$ is a constant. The function $\phi_{\delta}$ is then naturally extended to the rest of $M$.
When $\phi$ is replaced by $\phi_\delta$, we write the corresponding generalized Jang equation as
\begin{equation}\label{eq-deltaModifiedJang}
H_\delta(f)-K_\delta(f)=0.
\end{equation}
For each $\delta>0$ a blow-up solution $f_{\delta}$ of
\eqref{eq-deltaModifiedJang} exists. We aim to
show that for some subsequence
$\delta_{i}\rightarrow 0$, the surfaces $\Sigma_{\delta_{i}}=
\{t=f_{\delta_{i}}(x)\}$ converge smoothly away
from the boundary to a blow-up solution
$\Sigma=\{t=f(x)\}$ whose asymptotics at the boundary depend on the rates
at which $\phi$ and $\theta_{+}$ vanish.

We now establish the existence of blow-up solutions with the desired upper bound in the first case of Theorem \ref{thm4}.
The lower bound will be established in the next section.

\begin{prop}\label{prop1}
Suppose that $(M,g,k)$ is a smooth, asymptotically flat initial data set,
with outermost apparent horizon boundary $\partial M$ consisting of a
single future apparent horizon component.
Suppose further that $\phi$ is a smooth function,
strictly positive away from
$\partial M$ and satisfying
\eqref{10} and \eqref{14}, and that $\theta_+$
satisfies \eqref{12}.
Then for $b\ge -\frac{l-1}{2}$,
there exists a smooth solution $f$ of the generalized Jang
equation \eqref{7}, satisfying \eqref{11}, and
in a neighborhood of $\partial M$
\begin{align}\label{154}
\begin{split}
f &\leq \alpha\tau^{-b-\frac{l-1}{2}}+\beta\text{ }\text{ }\textit{ if }\text{ }\text{ }
b>-\frac{l-1}{2},\\
f &\leq -\alpha\log\tau+\beta\text{ }\text{ }\textit{ if }\text{ }\text{ }b=-\frac{l-1}{2},
\end{split}
\end{align}
for some positive constants $\alpha$ and $\beta$.
\end{prop}

\begin{proof} Let $f_{\delta,\varepsilon}$ denote solutions of the
$\varepsilon$-regularized generalized Jang equation with warping
factor $\phi_{\delta}$, as constructed in the previous section,
$$H_\delta(f_{\delta,\varepsilon})-K_\delta(f_{\delta,\varepsilon})
=\varepsilon f_{\delta,\varepsilon}.$$
We now proceed to construct an appropriate upper barrier
function $\psi$ for $f_{\delta,\varepsilon}$ which is independent of $\delta$
and $\varepsilon$. For this we will assume that \eqref{12} holds
for $\tau\in[0,\tau_{0}]$, with $\tau_{0}$ sufficiently small.
The upper barrier $\psi$ will be a function of $\tau$ alone,
namely $\psi=\psi(\tau)$. Recall that
\begin{align*}
H_\delta(\psi)-K_\delta(\psi)-\varepsilon\psi
&=\frac{\phi_{\delta}\psi''
+2\phi_{\delta}'\psi'}{(1+\phi_{\delta}^{2}\psi'^{2})^{3/2}}
+\frac{\phi_{\delta}\psi'\theta_{+}(S_{\tau})}
{\sqrt{1+\phi_{\delta}^{2}\psi'^{2}}}
-\frac{k_{nn}}{1+\phi_{\delta}^{2}\psi'^{2}}\\
&\qquad
-\left(1+\frac{\phi_{\delta}\psi'}{\sqrt{1+\phi_{\delta}^{2}\psi'^{2}}}\right)
Tr_{S_{\tau}}k-\varepsilon\psi.
\end{align*}
In the following, $\psi$ will be a decreasing function, that is, $\psi'<0$.
Hence, we have
\begin{align}\label{151}
\begin{split}
H_\delta(\psi)-K_\delta(\psi)-\varepsilon\psi
&=-\theta_{+}(S_{\tau})+\frac{\phi_{\delta}\psi''
+2\phi_{\delta}'\psi'}{(1+\phi_{\delta}^{2}\psi'^{2})^{3/2}}
-\frac{k_{nn}}{1+\phi_{\delta}^{2}\psi'^{2}}\\
&\qquad+\frac{H_{S_{\tau}}}{\sqrt{1+\phi_{\delta}^{2}\psi'^{2}}
({\sqrt{1+\phi_{\delta}^{2}\psi'^{2}}}-\phi_\delta \psi')}
-\varepsilon\psi.
\end{split}
\end{align}
Consider $$\psi(\tau)=\alpha \tau^{-a}+\beta,$$ for some constants
$a,\alpha,\beta\geq 0$.
In order for $|\phi\psi'|\rightarrow\infty$ as $\tau\rightarrow 0$,
we require
that $a>b-1$. A straightforward calculation shows that
$$\frac{|\phi_{\delta}\psi''
+2\phi_{\delta}'\psi'|}{(1+\phi_{\delta}^{2}\psi'^{2})^{3/2}}
+\frac{1}{1+\phi_{\delta}^{2}\psi'^{2}}
\le \frac{C}{a^2\alpha^2}(a+b+1)\tau^{2a-2b+1},$$
for $\tau$ sufficiently small.
By the assumption $\theta_{+}\geq c^{-1}\tau^{l}$, we obtain
\begin{equation*}\label{152}
H_\delta(\psi)-K_\delta(\psi)-\varepsilon\psi \leq
-c^{-1}\tau^{l}+ \frac{C}{a^2\alpha^2}(a+b+1)\tau^{2a-2b+1}.
\end{equation*}
Now set $a=b+\frac{l-1}{2}$, and notice that this automatically satisfies
the previous condition $a>b-1$. If $a>0$ and $\alpha$ is
chosen sufficiently large, then $\psi$ is a super solution,
$$H_\delta(\psi)-K_\delta(\psi)-\varepsilon\psi \leq 0.$$
If $a=0$, then setting
$$\psi(\tau)=-\alpha\log\tau+\beta$$
yields a similar result.

In order to obtain the estimate, recall that $|f_{\delta,\varepsilon}|$
will be uniformly bounded away
from the boundary, say at the surface corresponding to $\tau=\tau_{0}>0$.
This follows from the Harnack inequality and parametric estimates, as in \cite{SchoenYau}.
We may then choose $\beta$ sufficiently large independent of $\delta$
and $\varepsilon$ such that $\psi(\tau_{0})
\geq f_{\delta,\varepsilon}|_{S_{\tau_{0}}}$. A standard comparison argument
then shows that
\begin{equation}\label{153}
f_{\delta,\varepsilon}|_{S_{\tau}}\leq\psi(\tau)
\text{ }\text{ }\text{ for }\text{ }\text{ }\tau\in(0,\tau_{0}].
\end{equation}
As we have shown above, letting $\varepsilon\rightarrow 0$ yields a
blow-up solution $\Sigma_{\delta}$ to the generalized Jang
equation for each positive $\delta$. The solution $f_{\delta}$ must
of course also satisfy the asymptotics \eqref{153}.
%This
%establishes Theorem \ref{thm2} by choosing $b=0$.

Let us now extract a subsequence $\delta_{i}\rightarrow 0$ such that
the surfaces $\Sigma_{\delta_{i}}$ converge smoothly away from the
boundary to a solution
$\Sigma=\{t=f(x)\}$ of the generalized Jang equation. This may be
proved in the usual way \cite{SchoenYau}, making use of the parametric
estimates of Section \ref{sec2}
and the Harnack inequality of Section \ref{sec3}. Moreover
the asymptotics \eqref{153} still hold for $f$.
\end{proof}

Since the warping factor $\phi$ vanishes at the boundary,
we cannot say, without further analysis,
precisely how the solution behaves at the boundary. For instance,
although the solutions $\Sigma_{\delta_{i}}$ blow-up in the form of
a cylinder over $\partial M$,
it may be the case that as $\delta_{i}\rightarrow 0$ the blow-up
solutions become arbitrarily close to $\partial M\times\mathbb{R}$
and eventually (in the limit)
coincide or `stick' to the cylinder at certain points. In order to
prevent this, certain inequalities should hold between the vanishing
rates of $\phi$ and $\theta_{+}$;
this issue is the primary focus of the next section.

It turns out that the blow-up rates of \eqref{154} are not
the only asymptotics for the generalized
Jang equations. The next result provides other possibilities for upper barriers.

\begin{prop}\label{prop2}
Suppose that $(M,g,k)$ is a smooth, asymptotically flat initial
data set, with outermost apparent horizon boundary $\partial M$
consisting of a single future apparent horizon component.
Suppose further that $\phi$ is a smooth function,
strictly positive away from
$\partial M$ and satisfying
\eqref{10} and \eqref{14}, and that $\theta_+$
satisfies \eqref{12}.
Then for $b\ge 1/2$,
there exists a smooth solution $f$ of the generalized Jang
equation \eqref{7}, satisfying \eqref{11}, and
in a neighborhood of $\partial M$
\begin{align}\label{155}
\begin{split}
f &\leq \alpha\tau^{1-2b}+\beta\text{ }\text{ }\textit{ if }\text{ }\text{ }b>\frac{1}{2},\\
f &\leq -\alpha\log\tau+\beta\text{ }\text{ }\textit{ if }\text{ }\text{ }b=\frac{1}{2},
\end{split}
\end{align}
for some positive constants $\alpha$ and $\beta$.
\end{prop}

\begin{proof}
Recall the calculation in \eqref{151}. Our strategy will be to choose
the upper barrier function $\psi_{\delta}$, so that the second derivative
of $\psi_{\delta}$ will dominate all other terms in \eqref{151},
except $-\theta_+(S_\tau)$. To this end, consider
$$\psi_{\delta}''+\frac{2b}{\tau+\delta}\psi_{\delta}'=\lambda
\psi_{\delta}'$$
for some $\lambda$ to be determined. It is an easy
exercise to show that its solution is given by
\begin{equation}\label{156}
\psi_\delta(\tau)=\mu_{1}\int_{\tau+\delta}^1
\frac{e^{\lambda s}}{s^{2b}}\,ds+\mu_{2},
\end{equation}
for some positive $\mu_{1}$ and $\mu_{2}$ to be determined.
It is obvious that $\psi_\delta$ is monotonically decreasing.
Therefore
\begin{align*}
\psi_{\delta}''+2(\log\phi_{\delta})'\psi_{\delta}'
&=\psi_{\delta}''+\frac{2b}{\tau+\delta}\psi_{\delta}'
+2(\log\widetilde{\phi})'\psi_{\delta}'
\leq\psi_{\delta}''+\frac{2b}{\tau+\delta}\psi_{\delta}'-c\psi_{\delta}'\\
&=(\lambda-c)\psi_{\delta}'\leq\frac{\lambda}{2}\psi'_{\delta}<0
\end{align*}
for large $\lambda$. Hence, \eqref{151} becomes
\begin{align*}
H_{\delta}(\psi_{\delta})-K_{\delta}(\psi_{\delta})-\varepsilon\psi_{\delta}
&\leq -\theta_{+}
+\frac{\lambda\phi_{\delta}\psi_{\delta}'}
{2(1+\phi_{\delta}^{2}\psi_{\delta}'^{2})^{3/2}}
+c\phi_{\delta}^{-2}|\psi_{\delta}'|^{-2}\\
&\leq -c^{-1}\tau^{l}-\left(\frac{\lambda}{2}-c\right)
\phi_{\delta}^{-2}|\psi_{\delta}'|^{-2}< 0,
\end{align*}
if again $\lambda$ is large enough. It follows that
$\psi_{\delta}$ is a super solution.

Recall that $f_{\delta,\varepsilon}|_{\partial M}\leq c\varepsilon^{-1}$ by \eqref{135},
and note that
\begin{align*}
&\psi_{\delta}(0)\geq\alpha\delta^{1-2b}+\beta
\text{ }\text{ }\text{ if }\text{ }\text{ }b>\frac{1}{2},\\
&\psi_{\delta}(0)\geq-\alpha\log\delta+\beta\text{ }\text{ }
\text{ if }\text{ }\text{ }b=\frac{1}{2}.
\end{align*}
This shows that $\psi_{\delta}|_{\partial M}\geq
f_{\delta,\varepsilon}|_{\partial M}$ for sufficiently large $\alpha$,
if $\delta=\varepsilon^{(2b-1)^{-1}}$ when $b>1/2$,
and $\delta=e^{-\varepsilon^{-1}}$ when $b=1/2$. The maximum
principle then implies (if $\beta$ is chosen large enough) that
$f_{\delta,\varepsilon}\leq \psi_{\delta}$ for $\tau\in[0,\tau_{0}]$.
Let us now extract a subsequence $\varepsilon_{i}\rightarrow 0$
such that the surfaces $\Sigma_{\delta_{i},\varepsilon_{i}}$
converge smoothly away from the boundary to a solution
$\Sigma=\{t=f(x)\}$ of the generalized Jang equation.
This may be proved in the usual way as at the end of Proposition \ref{prop1},
even though here $\delta_{i}=\varepsilon_{i}^{(2b-1)^{-1}}$ when $b>1/2$, and
$\delta_{i}=e^{-\varepsilon_{i}^{-1}}$ when $b=1/2$.
It follows that the solution $f$ must satisfy the asymptotics \eqref{155}.
\end{proof}

Again, because the warping factor $\phi$ vanishes at the boundary,
we cannot say, without further analysis,
precisely how the solution behaves at the boundary. It might coincide or
`stick' to the cylinder over the boundary at certain points.
Nevertheless, the barrier construction yields
an upper bound for the asymptotics.

We may now compare the different asymptotics \eqref{154} and \eqref{155}.
Observe that $2b-1< b+\frac{l-1}{2}$ only when $b<\frac{l+1}{2}$, and
therefore the asymptotics in \eqref{155} only
improve the ones in \eqref{154} when $\frac{1}{2} \leq b < \frac{l+1}{2}$.

\section{Further Analysis of the Blow-Up Solutions}\label{sec6}

The purpose of this section is to obtain subsolutions for the
generalized Jang equation and to prove Theorem \ref{thm4}.
It will be shown that subsolutions exist with
the same asymptotics as described in Propositions \ref{prop1}
and \ref{prop2}. An appropriate comparison argument will then
be employed to
prove Theorem \ref{thm4}. As in the previous two sections, we
will assume for simplicity that the boundary of $M$ is an outermost
future apparent horizon with one component. All arguments to follow may be
extended to the general case in which the boundary is an outermost
apparent horizon, with several future and past horizon components. Moreover the
additional arguments needed for such an extension will be recorded at the end of
this section in the proof of Theorem \ref{thm4}.

Let $\psi(\tau)$ be a function of $\tau$ alone, which satisfies
$\psi'<0$ and
\begin{equation}\label{157}
|\phi\psi'|\rightarrow\infty\text{ }\text{ }\text{ as }\text{ }\text{ }\tau\rightarrow 0.
\end{equation}
Then according to \eqref{151} there exist bounded functions $c_{1}$ and $c_{2}$,
with $c_{1}\geq\frac{1}{2}$ for small $\tau$, such that
\begin{equation}\label{158}
H(\psi)-K(\psi)=-\theta_{+}(S_{\tau})
-\frac{c_{1}}{\tau^{2b}\psi'^{2}}\left(\frac{\psi''}{\psi'}+\frac{2b}{\tau}\right)
+\frac{c_{2}}{\tau^{2b}\psi'^{2}}.
\end{equation}

\begin{lemma}\label{lemma1}
$\operatorname{(1)}$
Suppose that $-\frac{l-1}{2}\leq b <\frac{l+1}{2}$, and that \eqref{12} and \eqref{14}
hold for $\tau\in[0,\tau_{0}]$. Then for sufficiently small $\tau_{0}$
there exist sub and supersolutions $\underline{\psi}$ and
$\overline{\psi}$ of the generalized Jang equation, satisfying
\eqref{157}, and such that
\begin{align}\label{159}\begin{split}
\alpha^{-1}\tau^{-b-\frac{l-1}{2}}+\beta^{-1} &\leq \underline{\psi},
\overline{\psi} \leq \alpha\tau^{-b-\frac{l-1}{2}}+\beta\text{ }\text{ }\text{ if }\text{ }\text{ }
-\frac{l-1}{2}<b<\frac{l+1}{2},\\
-\alpha^{-1}\log\tau+\beta^{-1} &\leq \underline{\psi},\overline{\psi} \leq -\alpha\log\tau+\beta\text{ }\text{ }\text{ if }
\text{ }\text{ }b=-\frac{l-1}{2},
\end{split}\end{align}
for some positive constants $\alpha$ and $\beta$. Moreover
\begin{equation*}\label{160}
H(\underline{\psi})-K(\underline{\psi})\geq\lambda\tau^{l},\text{ }\text{ }\text{ }\text{ }
H(\overline{\psi})-K(\overline{\psi})\leq-\frac{c^{-1}}{2}\tau^{l},
\end{equation*}
where the constant $\lambda>0$ may be chosen arbitrarily large and
$c^{-1}$ is as in \eqref{12}.

$\operatorname{(2)}$
Suppose that $\frac{1}{2}\leq b <\frac{l+1}{2}$, and that \eqref{14} and \eqref{15}
hold for $\tau\in[0,\tau_{0}]$. Then for sufficiently small $\tau_{0}$
there exist sub and supersolutions $\underline{\chi}$ and
$\overline{\chi}$ of the generalized Jang equation, satisfying
\eqref{157}, and such that
\begin{align*}
\alpha^{-1}\tau^{1-2b}+\beta^{-1} &\leq \underline{\chi},\overline{\chi} \leq \alpha\tau^{1-2b}+\beta\text{ }\text{ }\text{ if }\text{ }\text{ }
\frac{1}{2}<b<\frac{l+1}{2},\\
-\alpha^{-1}\log\tau+\beta^{-1} &\leq \underline{\chi},\overline{\chi} \leq -\alpha\log\tau+\beta\text{ }\text{ }\text{ if }\text{ }\text{ }b=\frac{1}{2},
\end{align*}
for some positive constants $\alpha$ and $\beta$. Moreover
\begin{equation*}
H(\underline{\chi})-K(\underline{\chi})\geq\lambda\tau^{\widetilde{l}},
\text{ }\text{ }\text{ }\text{ } H(\overline{\chi})-K(\overline{\chi})\leq-\lambda\tau^{\widetilde{l}},
\end{equation*}
where the constant $\lambda>0$ may be chosen arbitrarily large and $2b-1<\widetilde{l}\leq\min(l,2b)$.
\end{lemma}

\begin{proof}
$\operatorname{(1)}$
First consider the case $-\frac{l-1}{2}< b <\frac{l+1}{2}$ and take the
most obvious choice
$$\psi(\tau)=\alpha\tau^{-a}+\beta,$$
where $a$ is to be determined.
A calculation shows that
\begin{equation*}\label{161}
\frac{\psi''}{\psi'}+\frac{2b}{\tau}=\frac{2b-a-1}{\tau},
\end{equation*}
and $|\phi\psi'|=a\alpha\tau^{b-a-1}\rightarrow\infty$ as
$\tau\rightarrow 0$ if $a>b-1$. By applying \eqref{12} to \eqref{158}
we then have
\begin{equation*}\label{162}
H(\psi)-K(\psi)=-c_{3}\tau^{l}
-\frac{(2b-a-1)c_{1}}{a^{2}\alpha^{2}}\tau^{2a-2b+1}
+\frac{c_{2}}{a^{2}\alpha^{2}}\tau^{2a-2b+2},
\end{equation*}
where $c_{3}\geq c^{-1}>0$ is a bounded function.
Setting $a=b+\frac{l-1}{2}$ yields
\begin{equation*}\label{163}
H(\psi)-K(\psi)=-c_{3}\tau^{l}
+\frac{2(l+1-2b)c_{1}}{\alpha^{2}(2b+l-1)^{2}}\tau^{l}
+\frac{4c_{2}}{\alpha^{2}(2b+l-1)^{2}}\tau^{l+1}.
\end{equation*}
Thus if $\tau_{0}$ is sufficiently small, then by defining
$\underline{\psi}$ to have the structure of $\psi$ with
$\alpha$ sufficiently small
\begin{equation*}\label{164}
H(\underline{\psi})-K(\underline{\psi})\geq \alpha^{-1}\tau^{l},
\end{equation*}
and by defining $\overline{\psi}$ to have the structure
of $\psi$ with $\alpha$ sufficiently large
\begin{equation*}\label{165}
H(\overline{\psi})-K(\overline{\psi})\leq-\frac{c^{-1}}{2}\tau^{l}.
\end{equation*}

Now consider the case when $b=-\frac{l-1}{2}$, and take $$\psi(\tau)=-\alpha\log\tau+\beta.$$ Observe that
\begin{equation*}\label{166}
\frac{\psi''}{\psi'}+\frac{2b}{\tau}=-\frac{l}{\tau},
\end{equation*}
and $|\phi\psi'|=\alpha\tau^{-\frac{l+1}{2}}\rightarrow\infty$ as
$\tau\rightarrow 0$. By applying \eqref{12} to \eqref{158} we
then have
\begin{equation*}\label{167}
H(\psi)-K(\psi)=-c_{3}\tau^{l}+\frac{lc_{1}}{\alpha^{2}}\tau^{l}
+\frac{c_{2}}{\alpha^{2}}\tau^{l+1}.
\end{equation*}
Thus if $\tau_{0}$ is sufficiently small, then by defining
$\underline{\psi}$ to have the structure of $\psi$ with $\alpha$
sufficiently small
\begin{equation*}\label{168}
H(\underline{\psi})-K(\underline{\psi})\geq \alpha^{-1}\tau^{l},
\end{equation*}
and by defining $\overline{\psi}$ to have the structure of $\psi$
with $\alpha$ sufficiently large
\begin{equation*}\label{169}
H(\overline{\psi})-K(\overline{\psi})\leq-\frac{c^{-1}}{2}\tau^{l}.
\end{equation*}

$\operatorname{(2)}$
Assume that $\frac{1}{2}\leq b <\frac{l+1}{2}$, and let $\chi(\tau)$
be a function of $\tau$ alone satisfying \eqref{157}.
According to \eqref{158} and \eqref{15}
\begin{equation*}\label{170}
H(\chi)-K(\chi)=c_{4}\tau^{l}
-\frac{c_{1}}{\tau^{2b}\chi'^{2}}
\left(\frac{\chi''}{\chi'}+\frac{2b}{\tau}\right)
+\frac{c_{2}}{\tau^{2b}\chi'^{2}},
\end{equation*}
where $c_{4}$ is a bounded function. This suggests that we study
the ODE
\begin{equation}\label{171}
\frac{1}{\tau^{2b}\chi'^{2}}\left(\frac{\chi''}{\chi'}
+\frac{2b}{\tau}\right)=\mp \lambda \tau^{\widetilde{l}}.
\end{equation}
The choice of $\mp\lambda$ will be used when defining the sub and
supersolutions, respectively. In order to solve \eqref{171}
observe that
this equation is equivalent to
\begin{equation*}\label{172}
\left(\frac{\tau^{-4b}}{\chi'^{2}}\right)'=\pm2\lambda\tau^{\widetilde{l}-2b},
\end{equation*}
and hence
\begin{equation}\label{173}
\chi'=-\left(\frac{\pm 2\lambda}{\widetilde{l}+1-2b}\tau^{\widetilde{l}+1+2b}
+\Lambda^{2}\tau^{4b}\right)^{-\frac{1}{2}}
\end{equation}
for some constant $\Lambda$. We choose $\Lambda>0$, since if $\Lambda=0$ and $-\lambda$ is chosen in \eqref{171}, then \eqref{173} yields similar
asymptotics as in $\operatorname{(1)}$.
Notice also that the expression inside the square root is positive for
small $\tau$, since $4b<\widetilde{l}+1+2b$ and $\Lambda>0$. It follows that
\begin{equation*}\label{174}
\frac{\Lambda^{-1}}{2}\tau^{-2b}\leq -\chi'\leq 2\Lambda^{-1}\tau^{-2b}
\end{equation*}
for sufficiently small $\tau$. This shows that \eqref{157} holds. Furthermore
\begin{equation*}\label{175}
H(\chi)-K(\chi)=c_{4}\tau^{l} \pm c_{1}\lambda\tau^{\widetilde{l}}+c_{5}\Lambda^{2}\tau^{2b},
\end{equation*}
for some bounded function $c_{5}$. We define $\underline{\chi}$,
$\overline{\chi}$ to be the solutions of \eqref{171} constructed above and
corresponding to $-\lambda$, $+\lambda$ respectively. By choosing
$\lambda$ sufficiently large the desired result follows,
since $\widetilde{l}\leq\min(l,2b)$.
\end{proof}

The existence of two subsolutions with different asymptotics,
when $\frac{1}{2}\leq b<\frac{l+1}{2}$, indicates that there
will be two different blow-up solutions of
the generalized Jang equation, one corresponding to each of the
distinct asymptotics.  These solutions will arise from two
different sequences of solutions to the
regularized equation. More precisely, consider the generalized
Jang equation with $\phi=\tau^{b}\widetilde{\phi}$ replaced by
$\phi_{\delta}=(\tau+\delta)^{b}\widetilde{\phi}$, as in \eqref{eq-deltaModifiedJang}.
According to the proof of Proposition \ref{prop1},  for each $\delta>0$, there exists a blow-up solution
$t=f_{\delta}$ which asymptotically
approaches the cylinder $\partial M\times\mathbb{R}$ at a rate
given by \eqref{154}. A subsequence will then converge to a
blow-up solution of the generalized Jang equation as
$\delta\rightarrow 0$,
and this solution will satisfy the
asymptotics \eqref{16a}. The second sequence of solutions will arise from the
$\delta$-regularized generalized Jang equation, and will be
constructed to have finite values at $\partial M$. However, as
$\delta\rightarrow 0$ these boundary values will
become arbitrarily large, and a subsequence will converge to a
blow-up solution of the generalized Jang equation which satisfies
the asymptotics \eqref{16b}.

We first analyze the case of the solutions $t=f_{\delta}$ with asymptotics
given by \eqref{154}. Consider the $\delta$-regularized generalized
Jang equation
applied to a function $\psi(\tau)$ as in \eqref{158},
\begin{equation}\label{176}
H_{\delta}(\psi)-K_{\delta}(\psi)=-\theta_{+}(S_{\tau})
-\frac{c_{1}}{(\tau+\delta)^{2b}\psi'^{2}}\left(\frac{\psi''}{\psi'}
+\frac{2b}{\tau+\delta}\right)
+\frac{c_{2}}{(\tau+\delta)^{2b}\psi'^{2}}.
\end{equation}
Let $\underline{\psi}$ be the subsolution constructed in (1) of Lemma
\ref{lemma1}, and consider $$\underline{\psi}_{\delta}(\tau):=\underline{\psi}(\tau+\delta).$$
It is clear from the proof of Lemma \ref{lemma1} that
$\underline{\psi}_{\delta}$ is a subsolution of the $\delta$-regularized
generalized Jang equation,
\begin{equation}\label{177}
H_{\delta}(\underline{\psi}_{\delta})
-K_{\delta}(\underline{\psi}_{\delta})\geq\lambda(\tau+\delta)^{l}>0.
\end{equation}
We will show that $\underline{\psi}_{\delta}$ acts as a lower
barrier for the regularized solutions $f_{\delta}$.

\begin{prop}\label{prop3}
If $f_{\delta}|_{S_{\tau_{0}}}\geq\underline{\psi}_{\delta}(\tau_{0})$ then
$f_{\delta}\geq\underline{\psi}_{\delta}$ for all $\tau\in[0,\tau_{0}]$.
\end{prop}

\begin{proof}
Consider the function $w_{\delta}=f_{\delta}-\underline{\psi}_{\delta}$.
Since $\underline{\psi}_{\delta}(0)$ is finite and $f_{\delta}$ blows-up at
$\partial M$, we have that $w_{\delta}|_{\partial M}\geq 0$; moreover, by
assumption $w_{\delta}|_{S_{\tau_{0}}}\geq 0$. Suppose that $w_{\delta}$
attains
an interior negative minimum, then at that point
\begin{equation*}\label{178}
\nabla f_{\delta}=\nabla\underline{\psi}_{\delta},
\quad \nabla^{2}w_{\delta}\geq 0.
\end{equation*}
It follows that at the minimum point
\begin{equation*}\label{179}
(H_{\delta}(f_{\delta})-K_{\delta}(f_{\delta}))
-(H_{\delta}(\underline{\psi}_{\delta})-K_{\delta}(\underline{\psi}_{\delta}))
=\left(g^{ij}-\frac{\phi_{\delta}^{2}f_{\delta}^{i}f_{\delta}^{j}}
{1+\phi_{\delta}^{2}|\nabla f_{\delta}|^{2}}\right)
\frac{\phi_{\delta}\nabla_{ij}w_{\delta}}{\sqrt{1+\phi_{\delta}^{2}|\nabla f_{\delta}|^{2}}}\geq 0.
\end{equation*}
On the other hand, by \eqref{177}
\begin{equation*}\label{180}
(H_{\delta}(f_{\delta})-K_{\delta}(f_{\delta}))
-(H_{\delta}(\underline{\psi}_{\delta})-K_{\delta}(\underline{\psi}_{\delta}))\leq
-\lambda(\tau+\delta)^{l}<0.
\end{equation*}
This contradiction shows that $w_{\delta}\geq 0$ for all $\tau\in[0,\tau_{0}]$.
\end{proof}

%Recall that in this section, we have been assuming that $\partial M$ is
%an outermost future apparent horizon having one component. To this we
%will add one
%more restriction, namely that $\partial M$ is a \textit{pure} outermost
%future apparent horizon having one component. This means that the
%boundary is an outermost
%future apparent horizon ($\theta_{+}(\partial M)=0$) such that
%$\theta_{-}(\partial M)$ never vanishes.

We are now ready to prove the primary existence result for the
first asymptotics of Theorem \ref{thm4}.

\begin{theorem}\label{thm5}
Suppose that $(M,g,k)$ is a smooth, asymptotically flat initial data set,
with outermost apparent horizon boundary $\partial M$ consisting of a
single future apparent horizon component.
Suppose further that $\phi$ is a smooth function,
strictly positive away from
$\partial M$ and satisfying
\eqref{10} and \eqref{14}, and that $\theta_+$
satisfies \eqref{12} for $\tau\in[0,\tau_{0}]$.
If $-\frac{l-1}{2}\leq b <\frac{l+1}{2}$ then
there exists a smooth solution $f$ of the generalized
Jang equation \eqref{7}, satisfying \eqref{11}, and such that the
following estimates hold in a neighborhood of $\partial M$
\begin{align}\label{181}
\begin{split}
\alpha^{-1}\tau^{-b-\frac{l-1}{2}}+\beta^{-1} &\leq  f \leq \alpha\tau^{-b-\frac{l-1}{2}}+\beta\text{ }\text{ }\text{ if }\text{ }\text{ }
-\frac{l-1}{2}< b<\frac{l+1}{2},\\
-\alpha^{-1}\log\tau+\beta^{-1} &\leq  f \leq -\alpha\log\tau
+\beta\text{ }\text{ }\text{ if }\text{ }\text{ }b=-\frac{l-1}{2},
\end{split}
\end{align}
for some positive constants $\alpha$ and $\beta$.
\end{theorem}

\begin{proof}
According to the proof of Proposition \ref{prop1}, for each $\delta>0$, there exists a blow-up
solution $t=f_{\delta}$ of equation \eqref{7} with $\phi$ replaced
by $\phi_{\delta}$,
and this solution asymptotically approaches the cylinder
$\partial M\times\mathbb{R}$ at a rate given by \eqref{154}.
As in the proof of Proposition \ref{prop1}, there is a subsequence
$\delta_{i}\rightarrow 0$ such that these graphs converge smoothly
away from the boundary to a solution $t=f$ of the generalized Jang
equation \eqref{7}, which satisfies \eqref{11}. This implies that
there exists
a fixed value for $\underline{\psi}_{\delta_{i}}(\tau_{0})$,
independent of $\delta_{i}$,
such that $f_{\delta_{i}}|_{S_{\tau_{0}}}\geq\underline{\psi}_{\delta_{i}}(\tau_{0})$.
In fact by writing
$$\underline{\psi}_{\delta_{i}}(\tau)=\alpha_i^{-1}(\tau+\delta_{i})^{-b-\frac{l-1}2}+\beta_i^{-1},$$
we can achieve this by choosing fixed $\alpha_i=\alpha$ and $\beta_{i}=\beta$ for all $i$. Therefore $\underline{\psi}_{\delta_{i}}$ converges to a function
$\underline{\psi}$ satisfying the
estimates \eqref{159}. Moreover by Proposition \ref{prop3}, $f_{\delta_{i}}\geq\underline{\psi}_{\delta_{i}}$
for $\tau\in[0,\tau_{0}]$, and therefore $f$ satisfies \eqref{181}.
\end{proof}

Blow-up solutions of the generalized Jang equation satisfying
\eqref{16b} will now be constructed. These will arise from a sequence
of solutions
to the $\delta$-regularized equation with finite boundary values.
Assume that $\frac{1}{2}\leq b <\frac{l+1}{2}$ and that \eqref{15}
holds, and let $\underline{\chi}$, $\overline{\chi}$ be the functions
constructed in Lemma \ref{lemma1}. Set $$\underline{\chi}_{\delta}(\tau):=\underline{\chi}(\tau+\delta)
\quad\text{and}\quad
\overline{\chi}_{\delta}(\tau):=\overline{\chi}(\tau+\delta),$$ and note that these translated functions
are sub and supersolutions for the $\delta$-regularized generalized Jang
equation. In particular, from \eqref{176} and the proof of Lemma
\ref{lemma1} we have
\begin{equation}\label{183}
H_{\delta}(\underline{\chi}_{\delta})
-K_{\delta}(\underline{\chi}_{\delta})
\geq\lambda(\tau+\delta)^{\widetilde{l}},
\text{ }\text{ }\text{ }\text{ }H_{\delta}(\overline{\chi}_{\delta})
-K_{\delta}(\overline{\chi}_{\delta})
\leq-\lambda(\tau+\delta)^{\widetilde{l}}.
\end{equation}
Let $T_{\delta}=\delta^{1-2b}$ if $b>\frac{1}{2}$ and $T_{\delta}=-\log\delta$ if $b=\frac{1}{2}$, and observe that
$T_{\delta}\sim\underline{\chi}_{\delta}(0)\sim\overline{\chi}_{\delta}(0)$.
Consider now the following boundary value problem for the
$\delta$-regularized generalized Jang equation
\begin{equation}\label{184}
H_{\delta}(h_{\delta})-K_{\delta}(h_{\delta})=0
\text{ }\text{ }\text{ on }\text{ }\text{ }M,
\end{equation}
\begin{equation}\label{185}
h_{\delta}=T_{\delta}\text{ }\text{ }\text{ on }\text{ }\text{ }
\partial M,\quad
h_{\delta}\rightarrow 0\text{ }\text{ }\text{ as }\text{ }\text{ }
|x|\rightarrow\infty.
\end{equation}

\begin{prop}\label{prop4}
Suppose that $\partial M$ is an outermost future apparent horizon having one component,
and that \eqref{15} holds for $\tau\in[0,\tau_{0}]$.
If $\frac{1}{2}\leq b <\frac{l+1}{2}$, then there
exists a smooth solution of \eqref{184}, \eqref{185}, which is $C^{0}$ up to the boundary, and satisfies the following estimates in a
neighborhood of $\partial M$
\begin{align}\label{186}
\begin{split}
\alpha^{-1}(\tau+\delta)^{1-2b}+\beta^{-1} &\leq h_{\delta} \leq \alpha(\tau+\delta)^{1-2b}+\beta\text{ }\text{ }\text{ if }\text{ }\text{ }
\frac{1}{2}< b<\frac{l+1}{2},\\
-\alpha^{-1}\log(\tau+\delta)+\beta^{-1} &\leq h_{\delta}
\leq -\alpha\log(\tau+\delta)+\beta
\text{ }\text{ }\text{ if }\text{ }\text{ }b=\frac{1}{2},
\end{split}
\end{align}
for some positive constants $\alpha$ and $\beta$ independent of $\delta$.
\end{prop}

\begin{proof}
We first solve a boundary value problem for the
$\varepsilon\delta$-regularized generalized Jang equation:
\begin{equation}\label{184.1}
H_{\delta}(h_{\delta,\varepsilon})-K_{\delta}(h_{\delta,\varepsilon})
=\varepsilon (h_{\delta,\varepsilon}-\chi)
\text{ }\text{ }\text{ on }\text{ }\text{ }M,
\end{equation}
\begin{equation}\label{185.1}
h_{\delta,\varepsilon}=T_{\delta}\text{ }\text{ }\text{ on }\text{ }\text{ }
\partial M,\quad
h_{\delta,\varepsilon}\rightarrow 0\text{ }\text{ }\text{ as }\text{ }\text{ }
|x|\rightarrow\infty,
\end{equation}
where $\chi$ is a smooth function on $M$ such that $\chi=T_\delta$ on
$\partial M$, $\sup_{M}|\chi|\leq T_{\delta}$,
and $\chi\equiv 0$ in a neighborhood of spatial infinity. Moreover,
we require that $\chi$ is decreasing in $\tau$ near $\partial M$.

We now employ the continuity method to show that a unique
solution exists for \eqref{184.1} and \eqref{185.1}. This is similar to arguments used
for \eqref{147} and \eqref{148}.
To begin, observe that the maximum principle yields the bound
\begin{equation*}\label{186.1}
\sup_{M}|h_{\delta,\varepsilon}|\leq T_{\delta}+c\varepsilon^{-1}.
\end{equation*}
To establish boundary
gradient estimates, we construct subsolutions and supersolutions.

Fix a $\tau_0$ sufficiently small. For subsolutions, take $\underline{\chi}_{\delta, \varepsilon}$
as in Lemma \ref{lemma1}. Specifically, $\underline{\chi}_{\delta, \varepsilon}$
satisfies
\begin{equation*}
H_{\delta}(\underline{\chi}_{\delta, \varepsilon})
-K_{\delta}(\underline{\chi}_{\delta, \varepsilon})
\geq\lambda(\tau+\delta)^{\widetilde{l}},
\end{equation*}
for some positive $\lambda$.
Since this subsolution is a solution to a second order ODE, there are
two free parameters that may be chosen appropriately so that
\begin{equation*}\label{188}
\underline{\chi}_{\delta,\varepsilon}(0)=T_{\delta},
\text{ }\text{ }\text{ }\text{ }\underline{\chi}_{\delta,\varepsilon}(\tau_{0})
=-(T_{\delta}+c\varepsilon^{-1}),
\end{equation*}
and
\begin{equation*}
\underline{\chi}_{\delta,\varepsilon}\leq\chi
\text{ }\text{ }\text{ for }\text{ }\text{ }\tau\in[0,\tau_{0}].
\end{equation*}
This shows that $w_{\delta,\varepsilon}:=h_{\delta,\varepsilon}
-\underline{\chi}_{\delta,\varepsilon}\geq 0$
along the surfaces $\partial M$ and $S_{\tau_{0}}$.
Suppose that $w_{\delta,\varepsilon}$ achieves an interior negative minimum.
Then at that point, a standard comparison argument yields
\begin{equation*}\label{188.1}
[H_{\delta}(h_{\delta,\varepsilon})-K_{\delta}(h_{\delta,\varepsilon})]
-[H_{\delta}(\underline{\chi}_{\delta,\varepsilon})
-K_{\delta}(\underline{\chi}_{\delta,\varepsilon})]\geq 0.
\end{equation*}
On the other hand, according to \eqref{183}
and $w_{\delta,\varepsilon}<0$ at the point in question, we have
\begin{align*}\label{188.2}
\begin{split}
[H_{\delta}(h_{\delta,\varepsilon})-K_{\delta}(h_{\delta,\varepsilon})]
-[H_{\delta}(\underline{\chi}_{\delta,\varepsilon})
-K_{\delta}(\underline{\chi}_{\delta,\varepsilon})]
&\leq\varepsilon (h_{\delta,\varepsilon}-\chi)-\lambda(\tau+\delta)^{\widetilde{l}}\\
&=\varepsilon w_{\delta,\varepsilon}+\varepsilon(\underline{\chi}_{\delta,\varepsilon}-\chi)
-\lambda(\tau+\delta)^{\widetilde{l}}<0.
\end{split}
\end{align*}
It follows that $h_{\delta,\varepsilon}\geq\underline{\chi}_{\delta,\varepsilon}$ for all $\tau\in[0,\tau_{0}]$,
and this yields a lower bound for $\partial_{\tau}h_{\delta,\varepsilon}$.

For supersolutions, let $\widetilde{\chi}$ be defined in the same way as $\overline{\chi}$,
from the proof of Lemma \ref{lemma1}, except that \eqref{173} is replaced by
\begin{equation*}\label{187}
\widetilde{\chi}'=
\left(\frac{-2\lambda}{\widetilde{l}+1-2b}\tau^{\widetilde{l}+1+2b}
+\Lambda^{2}\tau^{4b}\right)^{-\frac{1}{2}}.
\end{equation*}
Then $\widetilde{\chi}_{\delta,\varepsilon}:=\widetilde{\chi}(\tau+\delta)$ satisfies
\begin{equation*}
H_{\delta}(\widetilde{\chi}_{\delta, \varepsilon})
-K_{\delta}(\widetilde{\chi}_{\delta, \varepsilon})
\leq-\lambda(\tau+\delta)^{\widetilde{l}}.
\end{equation*}
Choose the two parameters defining $\widetilde{\chi}_{\delta,\varepsilon}$ appropriately so that
\begin{equation*}\label{187.1}
\widetilde{\chi}_{\delta,\varepsilon}(0)=T_{\delta},\text{ }\text{ }\text{ }\text{ }\widetilde{\chi}_{\delta,\varepsilon}(\tau_{0})= T_{\delta}+c\varepsilon^{-1}.
\end{equation*}
Then $\overline{\chi}_{\delta,\varepsilon}\geq\chi$
for all $\tau\in[0,\tau_{0}]$. As above, a comparison argument can be
employed to show that $h_{\delta,\varepsilon}\leq\widetilde{\chi}_{\delta,\varepsilon}$ for all $\tau\in[0,\tau_{0}]$.
This yields an upper bound for $\partial_{\tau}h_{\delta,\varepsilon}$. A solution of \eqref{184.1}, \eqref{185.1} is now guaranteed, and uniqueness follows
from a maximum principle argument.

As in the proof of Theorem \ref{thm5}, we may now let $\varepsilon\rightarrow 0$ and extract a subsequence of solutions,
still denoted $h_{\delta,\varepsilon}$, which converges smoothly away from the horizon to a solution $h_{\delta}$ of \eqref{184}. Moveover
$$h_\delta\to 0\quad\text{as}\quad|x|\to\infty.$$ We need to show that
$$h_\delta =T_\delta \quad\text{on}\quad\partial M.$$
Note that the functions $h_{\delta,\varepsilon}$ are uniformly bounded,
independent of $\varepsilon$ (in fact independent of $\delta$ as well),
when restricted
to the surface $S_{\tau_{0}}$. It follows that $h_{\delta,\varepsilon}|_{S_{\tau_{0}}}<T_{\delta}$ for $\delta$ sufficiently small. We may then choose sub and supersolutions $\underline{\chi}_{\delta}\leq\overline{\chi}_{\delta}$,
satisfying \eqref{183}, the estimate \eqref{186} with constants $\alpha$ and $\beta$ independent of $\varepsilon$, and such that
\begin{equation*}\label{188.3}
\underline{\chi}_{\delta}(0)= h_{\delta,\varepsilon}|_{\partial M}=\overline{\chi}_{\delta}(0),\text{ }\text{ }\text{ }\text{ }
\underline{\chi}_{\delta}(\tau_{0})\leq h_{\delta,\varepsilon}|_{S_{\tau_{0}}}\leq\overline{\chi}_{\delta}(\tau_{0}),
\end{equation*}
and
\begin{equation*}
\underline{\chi}_{\delta}\leq\chi\leq\overline{\chi}_{\delta}
\text{ }\text{ }\text{ for }\text{ }\text{ }\tau\in[0,\tau_{0}].
\end{equation*}
A comparison argument, as in the first half of this proof, may now be used to establish \eqref{186} for $h_{\delta,\varepsilon}$.
This establishes $C^{0}$ and boundary gradient estimates, independent of $\varepsilon$. We may now apply Theorem 15.2 of \cite{GilbargTrudinger} to
obtain global $C^{1}$ bounds independent of $\varepsilon$. It follows that the limit $h_{\delta}$ is $C^{0}$ up to the boundary.

Lastly we may take the subsequential limit $h_{\delta}\rightarrow h$ as $\delta\rightarrow 0$, to find that the
functions $h_{\delta}$ are uniformly bounded, independent of $\delta$, when restricted
to the surface $S_{\tau_{0}}$. Then as in the previous paragraph, sub and supersolutions $\underline{\chi}_{\delta}$ and $\overline{\chi}_{\delta}$ may
be used to obtain the estimate \eqref{186} with constants $\alpha$ and $\beta$ independent of $\delta$.
\end{proof}

\begin{theorem}\label{thm6}
Suppose that $(M,g,k)$ is a smooth, asymptotically flat initial data set,
with outermost apparent horizon boundary $\partial M$ consisting of a
single future apparent horizon component.
Suppose further that $\phi$ is a smooth function,
strictly positive away from
$\partial M$ and satisfying
\eqref{10} and \eqref{14}, and that $\theta_+$
satisfies \eqref{15}.
If $\frac{1}{2}\leq b <\frac{l+1}{2}$ then
there exists a smooth solution $h$ of the generalized
Jang equation \eqref{7}, satisfying \eqref{11}, and such that the following
estimates hold in a neighborhood of $\partial M$
\begin{align}\label{189}
\begin{split}
\alpha^{-1}\tau^{1-2b}+\beta^{-1} &\leq  h \leq \alpha\tau^{1-2b}
+\beta\text{ }\text{ }\text{ if }\text{ }\text{ }\frac{1}{2}< b<\frac{l+1}{2},\\
-\alpha^{-1}\log\tau+\beta^{-1} &\leq h \leq -\alpha\log\tau+\beta
\text{ }\text{ }\text{ if }\text{ }\text{ }b=\frac{1}{2},
\end{split}
\end{align}
for some positive constants $\alpha$ and $\beta$.
\end{theorem}

\begin{proof}
By Proposition \ref{prop4}, for each $\delta>0$ there exists a
solution $t=h_{\delta}$ of boundary value problem \eqref{184},
\eqref{185}. As in the proof of Theorem \ref{thm5}, there is a subsequence
$\delta_{i}\rightarrow 0$ such that these graphs converge smoothly
away from the boundary to a solution $t=h$ of the generalized Jang
equation \eqref{7}, which
satisfies \eqref{11}. Since the constants $\alpha$ and $\beta$ in
\eqref{186}
are independent of $\delta$, the limit $h$ satisfies \eqref{189}.
\end{proof}

\begin{proof}[Proof of Theorem \ref{thm4}] This follows from
Theorems \ref{thm5} and \ref{thm6} if the boundary has one future
apparent horizon component. Similar arguments hold if the boundary has a single past apparent horizon
component.

Now consider the case of multiple components in which some components belong to
the category (1) and some belong to the category (2). The asymptotics of (1) arise from taking the limit as
$\varepsilon\rightarrow 0$ to obtain a blow-up solution, and then taking the limit as $\delta\rightarrow 0$. The asymptotics of (2)
arise from taking the limit as $\varepsilon\rightarrow 0$ to obtain a solution with finite boundary values \eqref{185}, and then taking the limit
as $\delta\rightarrow 0$ to obtain the blow-up solution. Moreover the relevant constructions may be performed in a neighborhood of each component,
so that this process of letting $\varepsilon\rightarrow 0$ first, and then taking the limit $\delta\rightarrow 0$, may be carried out
to obtain the desired result.
\end{proof}

We now give an example to show that the nice asymptotics described
in Theorems \ref{thm5} and \ref{thm6} can fail if $b=\frac{l+1}{2}$.
In particular, we will exhibit a solution of the generalized Jang equation
which sticks to the cylinder.

\begin{example}\label{ex1}
Let us consider the exterior region of the Schwarzschild spacetime
with metric
\begin{equation*}\label{190}
-\left(1-\frac{2m}{r}\right)dt^{2}+\left(1-\frac{2m}{r}\right)^{-1}dr^{2}
+r^{2}d\sigma^{2},
\end{equation*}
where $d\sigma^{2}$ is the round metric on $\mathbb{S}^{2}$. Let $t=f(r)$
be a radial graph, with induced metric
\begin{equation*}\label{191}
g=\left(\left(1-\frac{2m}{r}\right)^{-1}
-\left(1-\frac{2m}{r}\right)f'^{2}\right)dr^{2}+r^{2}d\sigma^{2}.
\end{equation*}
As is calculated in \cite{BrayKhuri}, the second fundamental form of
the graph is given by
\begin{equation*}\label{192}
k_{ij}=\frac{\phi\nabla_{ij}f+\phi_{i}f_{j}+\phi_{j}f_{i}}
{\sqrt{1+\phi^{2}|\nabla f|^{2}}},
\end{equation*}
where $\phi=\sqrt{1-\frac{2m}{r}}$ and the covariant derivatives are
calculated with respect to the
metric $g$. Thus $(M=\mathbb{R}^{3}-B_{2m}(0),g,k)$ forms an initial
data set for which the graph $t=f(r)$ is a solution of the generalized Jang
equation. We choose a solution $t=f(r)$ such that the
function $f$ is smooth up to the boundary $r=2m$. In particular,
$f(2m)<\infty$ and we may consider this as an example of a
Jang graph `sticking' to the cylinder.
For such an $f$,
we note that
\begin{equation*}\label{193}
g_{11}=\left(1-\frac{2m}{r}\right)^{-1}-\left(1-\frac{2m}{r}\right)f'^{2}
\sim\left(1-\frac{2m}{r}\right)^{-1}.
\end{equation*}
Therefore the distance to the boundary is given by
\begin{equation*}\label{194}
\tau=\int_{2m}^{r}\sqrt{g_{11}}\sim \sqrt{1-\frac{2m}{r}}.
\end{equation*}
This implies that $\phi\sim \tau$, or rather $b=1$.
We now compute the  null expansion of the coordinate spheres
with respect to the initial data metric $g$. A standard formula
\cite{BrayKhuri} yields
\begin{equation*}\label{195}
H_{S_{r}}=\frac{2\sqrt{g^{11}}}{r}\sim \sqrt{1-\frac{2m}{r}},
\end{equation*}
and the trace of the initial data $k$ over the coordinate spheres is given by
\begin{equation*}\label{196}
Tr_{S_{r}}k=-\frac{\phi\gamma^{ij}\Gamma_{ij}^{1}f'}
{\sqrt{1+\phi^{2}|\nabla f|^{2}}}=
\frac{2}{r}\frac{\phi g^{11}f'}{\sqrt{1+\phi^{2}g^{11}f'^{2}}}
\sim \left(1-\frac{2m}{r}\right)^{\frac32}.
\end{equation*}
It follows that
\begin{equation*}\label{197}
\theta_{\pm}=H_{S_{r}}\pm Tr_{S_{r}}k \sim \sqrt{1-\frac{2m}{r}}.
\end{equation*}
In terms of the notation used above, we have
$\theta_{\pm}\sim \tau$, and in particular,
$l=1$.   $\hfill\Box$
\end{example}

\bibliographystyle{amsplain}

\begin{thebibliography}{99}

\bibitem{AnderssonMetzger} L. Andersson, J. Metzger, \emph{The area of horizons and the trapped region}, Comm. Math. Phys.,
\textbf{290} (2009), no. 3, 941–-972.

\bibitem{anderssonmetzger} L. Andersson, J. Metzger, \emph{Curvature estimates for stable marginally trapped surfaces}, J. Differential Geom.,
\textbf{84} (2010), 231-265.

\bibitem{AnderssonEichmairMetzger} L. Andersson, M. Eichmair, J. Metzger, \emph{Jang's equation and its applications to marginally trapped surfaces},
Contemporary Mathematics, Complex Analysis and Dynamical Systems IV: Part 2. General Relativity, Geometry, and PDE (2011), 13-46.

\bibitem{Bray} H. Bray, \emph{Proof of the Riemannian Penrose conjecture using the positive mass theorem},
J. Differential Geom., \textbf{59} (2001), 177-267.

\bibitem{BrayKhuri} H. Bray, M. Khuri,
\emph{A Jang equation approach to the Penrose inequality}, Discrete Contin. Dyn. Syst., \textbf{27} (2010),
741-766. arXiv:0910.4785

\bibitem{BrayKhuri1} H. Bray, M. Khuri, \emph{P.D.E's which imply the Penrose conjecture}, Asian J. Math.,
\textbf{15} (2011), no. 4, 557-610. arXiv:0905.2622

\bibitem{DisconziKhuri} M. Disconzi, M. Khuri, \emph{On the Penrose inequality for charged black holes},
Class. Quantum Grav., \textbf{29} (2012), 245019. arXiv:1207.5484


\bibitem{EichmairMetzger} M. Eichmair, J. Metzger, \emph{Jenkins-Serrin type results for the Jang equation},
preprint, 2012, arXiv:1205.4301.

\bibitem{Galloway} G. Galloway, \emph{Rigidity of marginally trapped surfaces and the topology of black holes},
Commun. Anal. Geom., \textbf{16} (2008), 217-229.

\bibitem{GallowaySchoen} G. Galloway, R. Schoen, \emph{A generalization of Hawking's black hole topology theorem to higher dimensions},
Comm. Math. Phys., \textbf{266} (2006), no. 2, 571-576.

\bibitem{GilbargTrudinger} D. Gilbarg, N. Trudinger, \emph{Elliptic Partial Differential Equations of Second Order},
Springer-Verlag, New York, 1998.

\bibitem{HoffmanSpruck} D. Hoffman, J. Spruck, \emph{Sobolev and isoperimetric inequalities for riemannian submanifolds},
Communications on Pure and Applied Mathematics, \textbf{27} (1974), 715-727.

\bibitem{HuiskenIlmanen} G. Huisken, T. Ilmanen, \emph{The inverse mean curvature flow and the Riemannian Penrose
inequality}, J. Differential Geom., \textbf{59} (2001), 353-437.

\bibitem{Jang} P.-S. Jang, \emph{On the positivity of energy in General Relativity},
J. Math. Phys., \textbf{19} (1978), 1152-1155.

\bibitem{KhuriWeinstein} M. Khuri, G. Weinstein, \emph{Rigidity in the positive mass theorem with charge}, J. Math. Phys., \textbf{54} (2013), 092501. arXiv:1307.5499

\bibitem{MalecOMurchadha} E. Malec, N. \'{O} Murchadha, \emph{The Jang equation, apparent horizons, and the Penrose inequality},
Class. Q. Grav., \textbf{21} (2004), 5777-5787.

\bibitem{Metzger} J. Metzger, \emph{Blowup of Jang's equation at outermost marginally trapped surfaces}, Comm. Math. Phys.,
\textbf{294} (2010), no. 1, 61-–72.

\bibitem{SchoenYau0} R. Schoen, S.-T. Yau, \emph{On the positive mass conjecture in general relativity}, Comm. Math. Phys.,
\textbf{65} (1979), no. 1, 45-76.

\bibitem{SchoenYau} R. Schoen, S.-T. Yau, \emph{Proof of the positive mass theorem II},  Comm. Math. Phys., \textbf{79} (1981), no. 2, 231-260.

\bibitem{Witten} E. Witten, \emph{A new proof of the positive energy theorem},
Comm. Math. Phys., \textbf{80} (1981), no. 3, 381-402.


\end{thebibliography}

\end{document}